\documentclass[a4paper,11pt]{article}
\usepackage{amsmath, amsfonts, amscd, amssymb, amsthm, enumerate, xypic}

\def\NT{\text{NT}}

\newcommand{\Mat}{\operatorname{M}}
\newcommand{\Mats}{\operatorname{S}}
\newcommand{\Mata}{\operatorname{A}}

\newcommand{\GL}{\operatorname{GL}}
\newcommand{\Ker}{\operatorname{Ker}}
\newcommand{\Vect}{\operatorname{span}}
\newcommand{\im}{\operatorname{Im}}
\newcommand{\Co}{\operatorname{Cong}}
\newcommand{\urk}{\operatorname{urk}}
\newcommand{\tr}{\operatorname{tr}}

\newcommand{\rk}{\operatorname{rk}}

\renewcommand{\setminus}{\smallsetminus}

\def\F{\mathbb{F}}
\def\K{\mathbb{K}}

\def\calA{\mathcal{A}}
\def\calB{\mathcal{B}}

\def\calD{\mathcal{D}}

\def\calH{\mathcal{H}}

\def\calJ{\mathcal{J}}

\def\calL{\mathcal{L}}
\def\calM{\mathcal{M}}
\def\calN{\mathcal{N}}

\def\calR{\mathcal{R}}
\def\calS{\mathcal{S}}
\def\calT{\mathcal{T}}
\def\calU{\mathcal{U}}
\def\calV{\mathcal{V}}
\def\calW{\mathcal{W}}

\def\lcro{\mathopen{[\![}}
\def\rcro{\mathclose{]\!]}}

\theoremstyle{definition}
\newtheorem{Def}{Definition}
\newtheorem{Not}[Def]{Notation}

\theoremstyle{plain}
\newtheorem{theo}{Theorem}[section]
\newtheorem{prop}[theo]{Proposition}
\newtheorem{cor}[theo]{Corollary}
\newtheorem{lemma}[theo]{Lemma}
\newtheorem{claim}{Claim}
\newtheorem{step}{Step}

\theoremstyle{plain}

\theoremstyle{remark}
\newtheorem{Rems}{Remarks}
\newtheorem{Rem}[Rems]{Remark}

\title{Primitive spaces of matrices with upper rank two over the field with two elements}
\author{Cl\'ement de Seguins Pazzis\footnote{Universit\'e de Versailles Saint-Quentin-en-Yvelines, Laboratoire de Math\'ematiques
de Versailles, 45 avenue des Etats-Unis, 78035 Versailles cedex, France}
\footnote{e-mail address: dsp.prof@gmail.com}}

\begin{document}

\thispagestyle{plain}

\maketitle

\begin{abstract}
For fields with more than $2$ elements, the classification of the vector spaces of matrices with rank at most $2$ is already known.
In this work, we complete that classification for the field $\F_2$. We apply the results to obtain
the classification of triples of locally linearly dependent operators over $\F_2$, the classification of the
$3$-dimensional subspaces of $\Mat_3(\F_2)$ in which no matrix has a non-zero eigenvalue, and the classification of the
$3$-dimensional affine spaces that are included in the general linear group $\GL_3(\F_2)$.
\end{abstract}




\section{Introduction}

Let $n$ and $p$ be non-negative integers and $\K$ be an arbitrary field.
Given integers $i$ and $j$ such that $i \leq j$, we denote by $\lcro i,j\rcro$ the set of all integers $k$ such that $i \leq k \leq j$.

Given vector spaces $U$ and $V$ over $\K$, one denotes by $\calL(U,V)$ the space of all linear maps from $U$ to $V$.

We denote by $\Mat_{n,p}(\K)$ the space of all $n \times p$ matrices with entries in $\K$,
by $\Mat_n(\K)$ the space of all $n \times n$ matrices with entries in $\K$, by $\Mats_n(\K)$ (respectively, by $\Mata_n(\K)$, by $T_n^+(\K)$, by $\NT_n(\K)$) the space of all $n \times n$ symmetric (respectively, alternating, upper-triangular, and strictly upper-triangular) matrices with entries in $\K$.
Recall that an alternating matrix is a skew-symmetric matrix in which the diagonal entries equal zero.
The group of all invertible matrices
of $\Mat_n(\K)$ is denoted by $\GL_n(\K)$.
Given a matrix $M \in \Mat_{n,p}(\K)$, the entry of $M$ at the $(i,j)$-spot will be denoted by $m_{i,j}$ or, alternatively, by $M_{i,j}$.
Given a matrix $M \in \Mat_{n,p}(\K)$, a scalar $\lambda \in \K$ and distinct integers $i$ and $j$ in $\lcro 1,n\rcro$ (respectively, in $\lcro 1,p\rcro$),
the row operation $L_i \leftarrow L_i+\lambda L_j$ (respectively, the column operation $C_i \leftarrow C_i+\lambda C_j$), takes $M$ to the matrix with the same rows (respectively, columns) except the $i$-th one, which equals the sum of the $i$-th row (respectively, column) of $M$
with the product of the $j$-th row (respectively, column) of $M$ by $\lambda$.
One defines the row swap $L_i \leftrightarrow L_j$ (respectively, the column swap $C_i \leftrightarrow C_j$) likewise.

The upper rank of a linear subspace $\calV$ of $\Mat_{n,p}(\K)$ is the maximal rank for a matrix in $\calV$:
We denote it by $\urk(\calV)$. Two linear subspaces $\calV_1$ and $\calV_2$ of $\Mat_{n,p}(\K)$ are called equivalent,
and we write $\calV_1 \sim \calV_2$, when there are non-singular matrices $P \in \GL_n(\K)$ and $Q \in \GL_p(\K)$ such that
$\calV_2=P\,\calV_1\,Q$, meaning that $\calV_1$ and $\calV_2$ represent the same space of linear operators between finite-dimensional
vector spaces in different choices of bases of the source and target spaces.
If $n=p$ we say that $\calV_1$ and $\calV_2$ are similar, and we write $\calV_1 \simeq \calV_2$,
when the above condition holds with $Q=P^{-1}$.

A linear subspace $\calV$ of $\Mat_{n,p}(\K)$ with upper rank $r$ is called \textbf{primitive} when it satisfies the following conditions:
\begin{enumerate}[(i)]
\item No non-zero vector belongs to the kernel of every matrix of $\calV$.
\item The span of all the ranges of the matrices of $\calV$ is $\K^n$.
\item $\calV$ is not equivalent to a space $\calT$ of matrices of the form
$M=\begin{bmatrix}
H(M) & [?]_{n \times 1}
\end{bmatrix}$ where $\urk H(\calT) \leq r-1$.
\item $\calV$ is not equivalent to a space $\calT$ of matrices
of the form
$M=\begin{bmatrix}
H(M) \\
[?]_{1 \times p}
\end{bmatrix}$ where $\urk H(\calT) \leq r-1$.
\end{enumerate}
We say that $\calV$ is \textbf{reduced} whenever it satisfies conditions (i) and (ii),
and \textbf{semi-primitive} whenever it satisfies conditions (i), (ii) and (iii).
Note that those definitions are invariant under replacing $\calV$ with an equivalent subspace.
Thus, we can define primitive/semi-primitive/reduced operator spaces between finite-dimensional vector spaces.

\paragraph{}
Primitive spaces of bounded rank matrices were initially introduced by Atkinson and Lloyd \cite{AtkinsonPrim,AtkLloydPrim}
and later rediscovered  by Eisenbud and Harris \cite{EisenbudHarris}. In particular, Atkinson proved a general classification theorem \cite{AtkinsonPrim}
for all primitive subspaces of $\Mat_{n,p}(\K)$
with upper rank $r$ and for which $n>1+\dbinom{r}{2}$ or $p>1+\dbinom{r}{2}$, provided that $\K$ has more than $r$ elements.
In particular, for $r=2$, his theorem yields that, up to equivalence, the space $\Mata_3(\K)$ of all $3 \times 3$ alternating matrices
is the sole primitive matrix space with upper rank $2$ provided that the underlying field has more than $2$ elements.

Recent new insights have put the theory of primitive spaces back into the spotlight.
First of all, Atkinson's classification theorem for primitive spaces (and, more precisely, its generalization to semi-primitive spaces as given in \cite{dSPLLD2})
has been shown to yield a generalization of Gerstenhaber's theorem for fields with large cardinality, and we believe
that this new insight should help one have a better grasp of the structure of large spaces of nilpotent matrices \cite{dSPAtkinsontoGerstenhaber}.
On the other hand, semi-primitive matrix spaces are deeply connected to minimal locally linearly dependent spaces of operators,
and classification theorems for the former have been recently used to expand our understanding of the latter \cite{dSPLLD2}.

Considering the above, there is a renewed motivation for finding classification theorems for primitive spaces over small fields.
An earlier article of Beasley \cite{Beasley} contained some information on spaces of rank $2$ matrices
over $\F_2$ but fell short of giving a complete classification. It is the main purpose of the present work to achieve that classification.

\paragraph{}
For $(s,t)\in \lcro 0,n\rcro \times \lcro 0,p\rcro$, we denote by $\calR(s,t)$ the space of all $n \times p$ matrices of the form
$$\begin{bmatrix}
[?]_{s \times t} & [?]_{s \times (p-t)} \\
[?]_{(n-s) \times t} & [0]_{(n-s) \times (p-t)}
\end{bmatrix}.$$
When we use this notation, the number of rows and columns will always be obvious from the context.
If $s+t \leq \min(n,p)$, then $\calR(s,t)$ has upper rank $s+t$. In particular, if $n \geq 2$ and $p \geq 2$
the space $\calR(1,1)$ has upper rank $2$.

It is known that a space with upper rank $1$ is either equivalent to a subspace of
$\calR(1,0)$ or to a subspace of $\calR(0,1)$ (this classical result dates back to Issai Schur).
From there, one can determine the non-primitive reduced subspaces with upper rank $2$.
Indeed, let $\calV$ be such a space.
If $\calV$ is not semi-primitive, then it is equivalent to a subspace $\calV'$ of $\Mat_{n,p}(\K)$
in which every matrix splits up as $M=\begin{bmatrix}
H(M) & ?
\end{bmatrix}$ and $H(\calV') \subset \Mat_{n,p-1}(\K)$ has upper rank $1$;
Then, $\calV$ is equivalent to a subspace of $\calR(1,1)$ or to a subspace of $\calR(0,2)$, whether
$H(\calV')$ is equivalent to a subspace of $\calR(1,0)$ or to one of $\calR(0,1)$; In the second case $p\leq 2$ as $\calV$ is reduced, and hence $p=2$.
If the transpose of $\calV$ is not semi-primitive, then either $n=2$ or $\calV$ is
equivalent to a subspace of $\calR(1,1)$.
Conversely, if $p=2$ then $\calV$ cannot be semi-primitive (just delete the second column),
and the same holds if $\calV$ is equivalent to a subspace of $\calR(1,1)$ (delete the first column from the matrices of $\calR(1,1)$).

Thus:

\begin{prop}\label{basicsemiprimitiveprop}
Let $\calV$ be a reduced linear subspace of $\Mat_{n,p}(\K)$ with upper rank $2$.
Then, $\calV$ is semi-primitive if and only if $p>2$ and $\calV$ is not equivalent to a subspace of $\calR(1,1)$. \\
Moreover, $\calV$ is primitive if and only if $n>2$, $p>2$ and $\calV$ is not equivalent to a subspace of $\calR(1,1)$. \\
Thus, $\calV$ is semi-primitive if and only if $n=2$ or $\calV$ is primitive.
\end{prop}

The classification, up to equivalence, of the reduced linear subspaces of $\calR(1,1)$ is an easy exercise:

\begin{prop}
Assume that $n \geq 2$ and $p \geq 2$.
For every reduced linear subspace $\calV$ of $\calR(1,1) \subset \Mat_{n,p}(\K)$, there is a unique integer $r \in \lcro 0,\min(n-1,p-1)\rcro$
such that $\calV$ is equivalent to one (and only one) of the following spaces:
$$\Biggl\{\begin{bmatrix}
a & X^T & L \\
X & [0]_{r \times r} & [0]_{r \times (p-r-1)} \\
C & [0]_{(n-r-1) \times r} & [0]_{(n-r-1) \times (p-r-1)}
\end{bmatrix} \mid (a,X,C,L) \in \K \times \K^r \times \K^{n-r-1} \times \Mat_{1,p-r-1}(\K)
\Biggr\}$$
and
$$\Biggl\{\begin{bmatrix}
0 & X^T & L \\
X & [0]_{r \times r} & [0]_{r \times (p-r-1)} \\
C & [0]_{(n-r-1) \times r} & [0]_{(n-r-1) \times (p-r-1)}
\end{bmatrix} \mid (X,C,L) \in \K^r \times \K^{n-r-1} \times \Mat_{1,p-r-1}(\K)
\Biggr\}.$$
\end{prop}
From that point on, we shall focus on classifying primitive matrix spaces with upper rank $2$ over $\F_2$.

\begin{center}
In the rest of the article, we consider only the situation of a field $\K$ with two elements, denoted by $\F_2$.
\end{center}

It is known that for every field $\K$ the space $\Mata_3(\K)$ is primitive with upper rank $2$
(see \cite{AtkLloydPrim}), and this holds in particular for $\K=\F_2$.
Now, we introduce three additional examples of primitive spaces with upper rank $2$ over $\F_2$.
To simplify the discourse, it is convenient to describe such matrix spaces by generic matrices: Recall
that a generic matrix of a linear subspace $\calV$ of $\Mat_{n,p}(\K)$ is a matrix of the form
$\mathbf{x}_1A_1+\cdots+\mathbf{x}_m A_m$, where $\mathbf{x}_1,\dots,\mathbf{x}_m$ are independent indeterminates and
$(A_1,\dots,A_m)$ is a basis of $\calV$.

\begin{Not}
We define three linear subspaces of $\Mat_3(\F_2)$ by generic matrices in the following array:
\begin{center}
\begin{tabular}{| c || c | c | c |}
\hline
Space & $\calJ_3(\F_2)$ & $\calU_3(\F_2)$ & $\calV_3(\F_2)$ \\
\hline
Generic matrix & $\begin{bmatrix}
\mathbf{a} & \mathbf{c} & \mathbf{d} \\
0 & \mathbf{a}+\mathbf{b} & \mathbf{e} \\
0 &  0  & \mathbf{b}
\end{bmatrix}$ &
$\begin{bmatrix}
0 & \mathbf{a} & \mathbf{a}+\mathbf{c} \\
\mathbf{a} & 0 & \mathbf{b} \\
\mathbf{a}+\mathbf{b} &  \mathbf{c}  & 0
\end{bmatrix}$ &
$\begin{bmatrix}
0 & \mathbf{a} & \mathbf{c}+\mathbf{d} \\
\mathbf{c} & 0 & \mathbf{b} \\
\mathbf{a}+\mathbf{b} &  \mathbf{d}  & 0
\end{bmatrix}$ \\
\hline
\end{tabular}
\end{center}
\end{Not}

Alternatively, $\calJ_3(\F_2)$ can be seen as the space of all upper-triangular $3 \times 3$ matrices with trace zero,
whereas $\calV_3(\F_2)$ can be seen as the space of all matrices $M=(m_{i,j}) \in \Mat_3(\F_2)$ with diagonal zero
and $m_{1,2}+m_{2,3}+m_{3,1}=m_{3,2}+m_{2,1}+m_{1,3}=0$.
Note that $\calU_3(\F_2)$ is a linear subspace of $\calV_3(\F_2)$.

Given three scalars $a,b,c$ in $\F_2$ with $a+b+c=0$, one of them must be zero whence $abc=0$.
Computing the determinant, it is then obvious that every matrix in $\calJ_3(\F_2)$ is singular, and so is
every matrix in $\calV_3(\F_2)$ (or in $\calU_3(\F_2)$).

Now, we state our three main results:

\begin{prop}\label{U3V3structure}
The spaces $\calU_3(\F_2)$ and $\calV_3(\F_2)$ are primitive subspaces of $\Mat_3(\F_2)$
with upper rank $2$. Moreover, every non-zero matrix of one of those spaces has rank $2$.
\end{prop}

\begin{prop}\label{J3structure}
The space $\calJ_3(\F_2)$ is a primitive subspace of $\Mat_3(\F_2)$ with upper rank $2$. \\
A linear subspace of $\calJ_3(\F_2)$ is primitive with upper rank $2$ if and only if, for all $(a,b)\in (\F_2)^2$,
it contains at least one matrix of the form
$\begin{bmatrix}
a & ? & ? \\
0 & a+b & ? \\
0 & 0 & b
\end{bmatrix}$.
\end{prop}

\begin{theo}[Classification of primitive spaces with upper rank $2$ over $\F_2$]\label{classrank2F2}
Let $\calV$ be a primitive subspace of $\Mat_{n,p}(\F_2)$ with upper rank $2$.
Then, $n=p=3$ and exactly one of the following four conditions holds:
\begin{enumerate}[(i)]
\item $\calV$ is equivalent to a linear subspace of $\calJ_3(\F_2)$;
\item $\calV$ is equivalent to $\Mata_3(\F_2)$;
\item $\calV$ is equivalent to $\calU_3(\F_2)$;
\item $\calV$ is equivalent to $\calV_3(\F_2)$.
\end{enumerate}
\end{theo}

In Section \ref{structureofJ3section}, we shall also describe, up to equivalence, all the primitive spaces that are equivalent to a linear subspace of $\calJ_3(\F_2)$.

\begin{Rem}\label{transposeremark}
Note that if $\calV$ is a primitive subspace of $\Mat_{n,p}(\K)$, then its transpose is also primitive with the same upper rank.
It is obvious that each one of the spaces $\Mata_3(\F_2)$, $\calU_3(\F_2)$ and $\calV_3(\F_2)$ is equal to its transpose.
On the other hand, one sees that $\calJ_3(\F_2)$ is equivalent (and even similar) to its transpose by noting that
$\calJ_3(\F_2)^T$ is the space of all lower-triangular matrices of $\Mat_3(\F_2)$ with trace zero, and hence
it equals $K \calJ_3(\F_2) K^{-1}$ for the matrix $K:=\begin{bmatrix}
0 & 0 & 1 \\
0 & 1 & 0 \\
1 & 0 & 0
\end{bmatrix}$.
\end{Rem}

Let us immediately discuss some corollaries of the above results:

\begin{cor}
Let $\calV$ be a primitive $4$-dimensional subspace of $\Mat_{n,p}(\F_2)$.
Assume that $\calV$ is a rank--$2$ space, i.e.\ all its non-zero matrices have rank $2$.
Then, $\calV$ is equivalent to $\calV_3(\F_2)$.
\end{cor}

To see this, it suffices to show that $\calV$ cannot be equivalent to a subspace of $\calJ_3(\F_2)$.
This is easily obtained by noting that every $4$-dimensional subspace of $\calJ_3(\F_2)$
is a hyperplane of it, whence it must have a non-zero common vector with the $2$-dimensional subspace of all
matrices of the form $\begin{bmatrix}
0 & ? & ? \\
0 & 0 & 0 \\
0 & 0 & 0
\end{bmatrix}$, yielding a rank $1$ matrix in $\calV$.

\vskip 3mm
In \cite{Beasley}, Beasley stated without proof that the two $4$-dimensional subspaces
$$\Biggl\{\begin{bmatrix}
a & c & c  \\
d & a+b & c \\
d & d & b
\end{bmatrix} \mid (a,b,c,d)\in \F_2^4\Biggr\} \quad \text{and} \quad
\Biggl\{\begin{bmatrix}
a & 0 & c  \\
d & a+b & 0 \\
0 & c+d & b
\end{bmatrix} \mid (a,b,c,d)\in \F_2^4\Biggr\}$$
are inequivalent rank-$2$ spaces. However, although it is true that both are rank-$2$ spaces,
the above corollary shows that they are equivalent as one easily checks that both are primitive.

\vskip 3mm
Note finally that Theorem \ref{classrank2F2} yields a quick proof of a result of \cite{dSPclass} on the classification of
subspaces of singular matrices of $\Mat_3(\F_2)$ with dimension at least $5$. Indeed, given such a subspace $\calV$:
\begin{itemize}
\item Either $\calV$ is non-reduced, and hence it is equivalent to a subspace of $\calR(2,0)$ or to a subspace of $\calR(0,2)$.
\item Or $\calV$ is reduced and non-primitive, and hence it is equivalent to a subspace of $\calR(1,1)$;
in that case, as $\dim \calR(1,1)=5$, we see that $\calV$ is equivalent to $\calR(1,1)$ itself.
\item Or $\calV$ is primitive, and hence, as $\dim \calV \geq 5$, Theorem \ref{classrank2F2}
yields that $\calV$ is equivalent to a linear subspace of $\calJ_3(\F_2)$, and hence it is equivalent to $\calJ_3(\F_2)$
because $\dim \calJ_3(\F_2)=5$.
\end{itemize}

The article is laid out as follows: In Section \ref{structureofV3section}, we prove Proposition \ref{U3V3structure}.
Section \ref{proofmainclassification} is devoted to the proof of Theorem \ref{classrank2F2}.
In Section \ref{structureofJ3section}, we prove Proposition \ref{J3structure} and we classify all the primitive subspaces of $\calJ_3(\F_2)$.
In the last section, we use our results to classify triples of locally linearly dependent operators over $\F_2$ (Section \ref{LLDsection}), to classify
the $3$-dimensional linear subspaces of $\Mat_3(\F_2)$ in which no matrix has $1$ as eigenvalue (Section \ref{trivialspectrumsection}), and to classify the $3$-dimensional affine subspaces of $\Mat_3(\F_2)$ that are included in $\GL_3(\F_2)$ (Section \ref{affinenonsingularsection}).

\section{The structure of $\calU_3(\F_2)$ and $\calV_3(\F_2)$}\label{structureofV3section}

Remember that $\calV_3(\F_2)$ is the space of all matrices $M \in \Mat_3(\F_2)$
with diagonal zero and $m_{1,2}+m_{2,3}+m_{3,1}=0=m_{3,2}+m_{2,1}+m_{1,3}$, and that
$\calU_3(\F_2)$ is a hyperplane of $\calV_3(\F_2)$.

\begin{lemma}\label{V3rank}
Every non-zero matrix of $\calV_3(\F_2)$ has rank $2$.
\end{lemma}

\begin{proof}
We have already shown that every matrix of $\calV_3(\F_2)$ is singular.
Let $M \in \calV_3(\F_2)$ be with $\rk M \leq 1$. As $\tr M=0$, we deduce that $M^2=0$.
As the diagonal of $M$ is zero, this yields $m_{i,j}m_{j,k}=0$ for all distinct $i,j,k$ in $\lcro 1,3\rcro$.
In particular, among $m_{1,2},m_{2,3},m_{3,1}$, at most one entry equals $1$, and as their sum equals zero, we deduce that they are all zero.
Similarly, we obtain $m_{2,1}=m_{1,3}=m_{3,2}=0$, whence $M=0$.
\end{proof}

\begin{lemma}\label{U3trans}
For all $x \in (\F_2)^3 \setminus \{0\}$, one has $\dim \calU_3(\F_2)x=2$.
\end{lemma}

\begin{proof}
Denote by $\widehat{\calU_3(\F_2)}$ the space of all linear operators $M \in \calU_3(\F_2) \mapsto MX \in (\F_2)^3$, with $X \in (\F_2)^3$.
For $X=\begin{bmatrix}
x \\
y \\
z
\end{bmatrix}$ in $(\F_2)^3$, the operator $M \mapsto MX$ reads
$$M=\begin{bmatrix}
0 & a & a+c \\
a & 0 & b \\
a+ b & c & 0
\end{bmatrix} \mapsto \begin{bmatrix}
a(y+z)+c z\\
ax+b z \\
ax+b x+ cy
\end{bmatrix}=\begin{bmatrix}
y+z & 0 & z \\
x & z & 0 \\
x &  x  & y
\end{bmatrix} \times \begin{bmatrix}
a \\
b \\
c
\end{bmatrix}.$$
Thus, in a well chosen basis of $\calU_3(\F_2)$ and in the canonical basis of $\K^3$, the space
$\widehat{\calU_3(\F_2)}$ is represented by the matrix space
$$\Biggl\{\begin{bmatrix}
y+z & 0 & z \\
x & z & 0 \\
x &  x  & y
\end{bmatrix} \mid (x,y,z)\in (\F_2)^3 \Biggr\}.$$
By successively applying the row operation $L_3 \leftarrow L_3+L_2$ and the column operations $C_2 \leftrightarrow C_1$
and $C_3 \leftrightarrow C_2$, this space is seen to be equivalent to $\calU_3(\F_2)$.
As this space has dimension $3$ and every non-zero matrix of  $\calU_3(\F_2)$ has rank $2$, we deduce that $\dim \calU_3(\F_2)x=2$
for all non-zero vectors $x \in (\F_2)^3$.
\end{proof}

As $\calU_3(\F_2)^T=\calU_3(\F_2)$, it follows that $\dim \calU_3(\F_2)^T x=2$ for all non-zero vectors $x \in (\F_2)^3$,
whence $\calU_3(\F_2)$ is reduced.
It also follows from Lemma \ref{U3trans} that $\calU_3(\F_2)$ is not equivalent to a linear subspace of $\calR(1,1)$.
Therefore, $\calU_3(\F_2)$ is primitive and it ensues that $\calV_3(\F_2)$ is also primitive since it contains $\calU_3(\F_2)$ and
shares the same upper rank. Thus, Proposition \ref{U3V3structure} is established.

\section{Proof of the main classification theorem}\label{proofmainclassification}

This section is devoted to the proof of our main classification theorem, that is Theorem \ref{classrank2F2}.
First of all, we shall prove that cases (i) to (iv) are pairwise incompatible. Then, we will examine two special cases with $n=p=3$.
Afterwards, we will prove that $\Mat_{n,p}(\F_2)$ has a primitive subspace with upper rank $2$ only if $n=p=3$.
Finally, we will classify the primitive subspaces of $\Mat_3(\F_2)$ with upper rank $2$.

\subsection{Incompatibility between Cases (i) to (iv)}

To see that no two cases of Cases (i) to (iv) in Theorem \ref{classrank2F2} can occur simultaneously,
note that, whenever $\calV$ falls into one of Cases (ii) to (iv), we have
$\dim \calV x \geq 2$ for every non-zero vector $x \in (\F_2)^3$, which rules Case (i) out.

Case (iv) is incompatible with Cases (ii) and (iii) because $\dim \calV_3(\F_2)=4$,
whereas $\dim \calU_3(\F_2)=\dim \Mata_3(\F_2)=3$.

Finally, Case (iii) is incompatible with Case (ii) because
if Case (ii) holds, for every $M \in \calV \setminus \{0\}$, the non-zero vector $x$ of $\Ker M$
satisfies $\calV x= \im M$ (indeed, in the special case when $\calV=\Mata_3(\F_2)$, we have
$\im M=\{x\}^\bot=\calV x$ where $\bot$ refers to the canonical bilinear form $(X,Y) \mapsto X^TY$ on $(\F_2)^3$),
whereas this is not always the case for the space
$\calU_3(\F_2)$. Indeed, the matrix $M=\begin{bmatrix}
0 & 1 & 1 \\
1 & 0 & 0 \\
1 & 0 & 0
\end{bmatrix}$ belongs to $\calU_3(\F_2)$, the non-zero vector $x=\begin{bmatrix}
0 & 1 & 1
\end{bmatrix}^T$ belongs to its kernel, but we have $\calU_3(\F_2) x=\bigl\{\begin{bmatrix}
a & b & a
\end{bmatrix}^T \mid (a,b) \in (\F_2)^2\bigr\}$, which is obviously unequal to $\im M$.

\subsection{Two basic lemmas}

\begin{lemma}\label{transrank1}
Let $\calV$ be a primitive subspace of $\Mat_3(\F_2)$ with upper rank $2$.
Assume that there is a non-zero vector $x \in (\F_2)^3$ such that $\dim \calV x\leq 1$.
Then $\calV$ is equivalent to a subspace of $\calJ_3(\F_2)$.
\end{lemma}

\begin{proof}
Without loss of generality, we may assume that every matrix $M$ of $\calV$
splits up as
$$M=\begin{bmatrix}
a(M) & [?]_{1 \times 2} \\
[0]_{2 \times 1} & K(M)
\end{bmatrix} \quad \text{with $a(M) \in \F_2$ and $K(M) \in \Mat_2(\F_2)$.}$$
As every matrix of $\calV$ is singular, $a(M)=0$ whenever $K(M)$ is non-singular.

Assume that $K(\calV)$ is inequivalent to a subspace of $T_2^+(\F_2)$.
In particular, $K(\calV)$ must contain a non-singular matrix (by the classification of spaces with upper rank $1$).
Then, we have some $M_0 \in \calV$ such that $K(M_0)$ is non-singular and hence $a(M_0)=0$.
For all $M \in \calV$, if $K(M)=0$ then $K(M+M_0)$ is non-singular and hence $a(M+M_0)=0$,
which yields $a(M)=0$. It follows that there is a linear form $\varphi : K(\calV) \rightarrow \F_2$
such that $a(M)=\varphi(K(M))$ for all $M \in \calV$.
As $a \neq 0$ (because $\calV$ is reduced), we see that $K(\calV)$ cannot be spanned by its non-singular matrices.
If $K(\calV)$ were a hyperplane of $\Mat_2(\F_2)$, then it would be equivalent to
$T_2^+(\F_2)$ or to $\Mats_2(\F_2)$, whether its orthogonal subspace for $(A,B) \mapsto \tr(AB)$
contained a rank $1$ matrix or not. However, $\Mats_2(\F_2)$ is spanned by its non-singular elements, and so does
$\Mat_2(\F_2)$, whence $\dim K(\calV) \leq 2$. Moreover, as
$K(\calV)$ is inequivalent to a subspace of $T_2^+(\F_2)$, we have
$K(\calV)y=(\F_2)^2$ for all non-zero vectors $y \in (\F_2)^2$; It ensues that $\dim K(\calV)=2$
and that all the non-zero matrices of $K(\calV)$ are non-singular, contradicting the fact that
$K(\calV)$ is not spanned by its non-singular matrices.

Thus, $K(\calV)$ is actually equivalent to a subspace of $T_2^+(\F_2)$. Therefore,
no generality is lost in assuming that $\calV$ is actually a linear subspace of $T_3^+(\F_2)$.
For $M \in \calV$, denote by $\delta(M)=\begin{bmatrix}
m_{1,1} & m_{2,2} & m_{3,3}
\end{bmatrix}^T \in (\F_2)^3$ its diagonal vector.
Then, $\delta(\calV)$ is a linear subspace of $(\F_2)^3$
that does not contain the vector
$\begin{bmatrix}
1 & 1 & 1
\end{bmatrix}^T$ (since no matrix of $\calV$ is invertible). Thus, $\delta(\calV)$ is included in a hyperplane with the same property.
Moreover, since $\calV$ is primitive, $\delta(\calV)$ is included in none of the three canonical hyperplanes
(with equations $x_1=0$, $x_2=0$ and $x_3=0$, respectively).
The only remaining hyperplane which does not contain $\begin{bmatrix}
1 & 1 & 1
\end{bmatrix}^T$ is the one defined by the equation $x_1+x_2+x_3=0$, whence every matrix of $\calV$ has trace $0$.
We conclude that $\calV$ is a linear subspace of $\calJ_3(\F_2)$.
\end{proof}

\begin{lemma}\label{rank1}
Let $\calV$ be a primitive subspace of $\Mat_3(\F_2)$ with upper rank $2$.
Assume that $\calV$ contains a rank $1$ matrix.
Then, $\calV$ is equivalent to a subspace of $\calJ_3(\F_2)$.
\end{lemma}

\begin{proof}
Without loss of generality, we may assume that $\calV$ contains the elementary matrix
$E_{1,1}$ (with entry $1$ at the $(1,1)$-spot, and zero entries everywhere else). Then, we split every $M \in \calV$ as
$$M=\begin{bmatrix}
a(M) & R(M) \\
S(M) & K(M)
\end{bmatrix} \quad \text{with $a(M) \in \F_2$ and $K(M) \in \Mat_2(\F_2)$.}$$
We contend that every matrix of $K(\calV)$ is singular. Indeed, if we let $M \in \calV$,
then both matrices $M$ and $M+E_{1,1}$ belong to $\calV$, and therefore
$$0=\det(M+E_{1,1})-\det M=\det K(M).$$
It follows that $\urk K(\calV) \leq 1$.
Then, there are two cases to consider:
\begin{itemize}
\item Either there is a non-zero vector of $(\F_2)^2$ on which all the matrices of $K(\calV)$ vanish;
in this case we find a non-zero vector $x$ of $(\F_2)^3$ for which $\dim \calV x \leq 1$, and
Lemma \ref{transrank1} shows that $\calV$ is equivalent to a linear subspace of $\calJ_3(\F_2)$.
\item Or the non-zero matrices of $K(\calV)$ have the same range, whence there is a non-zero
vector $x \in (\F_2)^3$ for which $\dim \calV^T x \leq 1$. As $\calV^T$
is primitive with upper rank $2$, we deduce from Lemma \ref{transrank1} that it is equivalent to a linear subspace of
$\calJ_3(\F_2)$. However, we have seen in Remark \ref{transposeremark} that $\calJ_3(\F_2)^T$ is equivalent to $\calJ_3(\F_2)$,
whence $\calV$ is equivalent to a linear subspace of $\calJ_3(\F_2)$.
\end{itemize}
\end{proof}

\subsection{Basic identities}\label{identitysection}

In the rest of the proof, we let $\calV$ be a primitive subspace of $\Mat_{n,p}(\F_2)$ with upper rank $2$.
Note that $n\geq 3$ and $p \geq 3$.

As $\calV$ contains a rank $2$ matrix and as such a matrix is equivalent to
$$J_2:=\begin{bmatrix}
I_2 & [0]_{2 \times (p-2)} \\
[0]_{(n-2) \times 2} & [0]_{(n-2) \times (p-2)}
\end{bmatrix},$$
we lose no generality in assuming that $\calV$ contains $J_2$.

We split every matrix $M$ of $\calV$ up as
$$M=\begin{bmatrix}
A(M) & C(M) \\
B(M) & D(M)
\end{bmatrix}$$
along the same pattern as $J_2$.

Let $i \in \lcro 3,n\rcro$ and $j \in \lcro 3,p\rcro$.
The $3$ by $3$ sub-matrix of $M$ obtained by selecting row indices in $\{1,2,i\}$ and column indices in $\{1,2,j\}$
is singular since $\rk M \leq 2$, and on the other hand its determinant reads
$$\bigl(\det A(M)\bigr)\,D(M)_{i-2,j-2}-B(M)_{i-2} \,\widetilde{A(M)}\, C(M)_{j-2},$$
where $\widetilde{N}$ denotes the transpose of the comatrix
of the square matrix $N$, and $B(M)_{i-2}$ and $C(M)_{j-2}$ respectively denote the $(i-2)$-th row of $B(M)$ and
the $(j-2)$-th column of $C(M)$. Varying $i$ and $j$ then yields the matrix identity
\begin{equation}\label{fundid}
\bigl(\det A(M)\bigr) D(M)=B(M)\, \widetilde{A(M)}\, C(M),
\end{equation}
Note that $N \mapsto \widetilde{N}$ is linear on $\Mat_2(\F_2)$.
Moreover
\begin{equation}\label{fundid2}
\forall M \in \calV, \; D(M)=0 \Rightarrow B(M)\,C(M)=0.
\end{equation}
To see this, it suffices to apply identity \eqref{fundid} to both matrices $M$ and $M+J_2$.

\subsection{The proof that $n=p=3$}

Now, we prove the following result:

\begin{prop}\label{downto3}
Let $\calV$ be a primitive subspace of $\Mat_{n,p}(\F_2)$ with upper rank $2$. Then, $n=p=3$.
\end{prop}

The proof has several steps. First of all, we lose no generality in assuming that $\calV$ contains $J_2$, as in the preceding section.
Our first step establishes an important relationship between the matrices $B(M)$ and $D(M)$, for $M$ in $\calV$:

\begin{step}\label{columnclaim}
Let $M \in \calV$. Denote by $B_1(M)$ and $B_2(M)$ the columns of $B(M)$.
Then,
$$\rk \begin{bmatrix}
B_1(M) & D(M)
\end{bmatrix}\leq 1 \quad \text{and} \quad \rk \begin{bmatrix}
B_2(M) & D(M)
\end{bmatrix}\leq 1.$$
\end{step}

\begin{proof}
Take two distinct indices $i_1$ and $i_2$ in $\lcro 3,n\rcro$,
two distinct indices $j_1$ and $j_2$ in $\lcro 2,p\rcro$, and denote by
$\Delta(M)$ the $3 \times 3$ sub-matrix of $M$ obtained by selecting the row indices in $\{1,i_1,i_2\}$
and the column indices in $\{1,j_1,j_2\}$. Then, we see that
$\det \Delta(M+J_2)-\det \Delta(M)$ is the determinant of the $2 \times 2$ submatrix
of $M$ obtained by selecting row indices in $\{i_1,i_2\}$ and column indices in $\{j_1,j_2\}$.
As $\det \Delta(M+J_2)=0=\det \Delta(M)$, we deduce that
$\rk \begin{bmatrix}
B_2(M) & D(M)
\end{bmatrix}\leq 1$ by varying $i_1,i_2,j_1,j_2$.
The first inequality is proved in a similar fashion.
\end{proof}

As an immediate corollary, we deduce:

\begin{step}
The upper rank of $D(\calV)$ is less than or equal to $1$.
\end{step}

It follows that either all the non-zero matrices of $D(\calV)$ have the same kernel, or all of them have the same range.

\begin{step}\label{claimcaseDVnonzero}
Assume that $D(\calV) \neq \{0\}$.
If all the non-zero matrices of $D(\calV)$ have the same range (respectively, the same kernel), then $n=3$ (respectively, $p=3$).
\end{step}

\begin{proof}
Note that if all the non-zero matrices of the form $\begin{bmatrix}
B_1(M) & D(M)
\end{bmatrix}$ have the same kernel, then this kernel cannot be $\{0\} \times (\F_2)^{p-2}$ as $D(\calV) \neq \{0\}$;
then, as this kernel must have dimension $p-2$, it must contain a vector of $(\F_2)^{p-1} \setminus \bigl(\{0\} \times (\F_2)^{p-2}\bigr)$,
which yields a column matrix $X \in \K^{p-2}$ such that $B_1(M)=D(M)X$ for all $M \in \calV$.

Assume that all the non-zero matrices of $D(\calV)$ have the same range, denoted by $\calD$.
By Step \ref{columnclaim}, if all the non-zero matrices of the form $\begin{bmatrix}
B_1(M) & D(M)
\end{bmatrix}$ did not have the same range, then they would all have the same kernel -- owing the classification of matrix spaces with upper rank
at most $1$ -- and hence the above remark shows that $B_1(M) \in \im D(M) \subset \calD$ for all
$M \in \calV$. If all those matrices have the same range, it must be $\calD$ because $D(\calV) \neq \{0\}$. In any case, we obtain
$B_1(M) \in \calD$ for all $M \in \calV$.
Similarly, one obtains $B_2(M) \in \calD$ for all $M \in \calV$.
As $\calV$ is reduced, we deduce that $n=3$.

Using $\calV^T$ instead of $\calV$, we deduce that $p=3$ if
all the non-zero matrices of $D(\calV)$ have the same kernel.
\end{proof}

\begin{step}\label{dimatmost1claim}
One has $\dim D(\calV) \leq 1$.
\end{step}

\begin{proof}
Assume that $\dim D(\calV)>1$. Assume also that all the non-zero matrices of $D(\calV)$ have the same
kernel. Then, $p=3$. Moreover, by Step \ref{claimcaseDVnonzero},
all the non-zero matrices of the form $\begin{bmatrix}
B_2(M) & D(M)
\end{bmatrix}$ cannot have the same range, which, by Step 1, yields a scalar $\mu$ such that
$B_2(M)=\mu\, D(M)$ for all $M \in \calV$.
Similarly, one finds $\lambda \in \F_2$ for which $B_1(M)=\lambda\, D(M)$ for all $M \in \calV$.
Performing the column operations $C_1 \leftarrow C_1 -\lambda C_3$ and $C_2 \leftarrow C_2-\mu C_3$
changes none of the above assumptions and reduces the situation to the one where $B(M)=0$ for all $M \in \calV$.
Note that every matrix $M$ of $\calV$ then splits up as
$$M=\begin{bmatrix}
A(M) & [?]_{2 \times 1} \\
[0]_{(n-2) \times 2} & D(M)
\end{bmatrix},$$
whence $\rk A(M)=2 \Rightarrow D(M)=0$.

Let $M \in \calV$ be such that $A(M)=0$.
Then, $A(M+J_2)=I_2$, whence $0=D(M+J_2)=D(M)$. This yields a linear map
$\varphi : A(\calV) \rightarrow (\F_2)^{n-2}$ such that $D(M)=\varphi(A(M))$ for all $M \in \calV$,
and $\varphi$ vanishes at every rank $2$ matrix of $A(\calV)$.
Note that $\dim \Ker \varphi \geq 1$ and $\rk \varphi \geq 2$, whence $\dim A(\calV) \geq 3$.
If $\dim A(\calV)=4$, then $A(\calV)=\Mat_2(\F_2)$ is spanned by its rank $2$ elements, which leads to
$\varphi=0$. Thus, $\dim A(\calV)=3$, $\rk \varphi=2$ and $\dim \Ker \varphi=1$.
But again, we find a contradiction by noting that every linear hyperplane of $\Mat_2(\F_2)$
contains several rank $2$ matrices (this is obvious as such a hyperplane must be equivalent to $\Mats_2(\F_2)$ or to $T_2^+(\F_2)$,
as we have already explained in the course of the proof of Lemma \ref{transrank1}).
Therefore, the non-zero matrices of $D(\calV)$ cannot share the same kernel.

Similarly, by working with $\calV^T$, we see that the non-zero matrices of $D(\calV)$ cannot share the same range.
Therefore, we have contradicted the fact that $D(\calV)$ has upper rank $1$.
\end{proof}

From there, we can complete our proof of Proposition \ref{downto3}:

\begin{step}
One has $n=p=3$.
\end{step}

\begin{proof}
If $D(\calV)\neq \{0\}$, then $\dim D(\calV)=1$, whence all the non-zero matrices of $D(\calV)$ have the same range and the same kernel
(there is only one such matrix!), and Step \ref{claimcaseDVnonzero} yields $n=p=3$.

In the rest of the proof, we assume that $D(\calV)=\{0\}$.
As $\calV$ is reduced, we have $B(\calV) \neq \{0\}$ and $C(\calV) \neq \{0\}$.
For all $M \in \calV$, we know from identity \eqref{fundid2} that $B(M)C(M)=0$.
In particular $B(M)=0$ whenever $\rk C(M)=2$.
Assume that some matrix $M_0$ is such that $\rk C(M_0)=2$.
Then, $B(M_0)=0$. For every $M \in \calV$, we find $B(M) C(M+M_0)=B(M+M_0)C(M+M_0)=0$, whence
$B(M)C(M_0)=B(M)C(M+M_0)-B(M)C(M)=0$, which leads to $B(M)=0$, contradicting our assumptions.
Thus, $\urk C(\calV)=1$, and similarly $\urk B(\calV)=1$.

If all the matrices of $C(\calV)$ have the same kernel, we obtain that $p=3$ since $\calV$ is reduced.
Assume now that all the non-zero matrices of $C(\calV)$ have the same range. Without loss of generality, we
can assume that this range is $\F_2 \times \{0\}$.
Note that $\dim C(\calV)=p-2$ because of condition (i) in the definition of a primitive space.
Then, for all $M \in \calV$, we write $C(M)=\begin{bmatrix}
L(M) \\
[0]_{1 \times (p-2)}
\end{bmatrix}$.
If we let $M \in \calV$, then identity \eqref{fundid2} yields $B_1(M)L(M)=0$, whence either $B_1(M)=0$ or $L(M)=0$.
As $\calV$ is not the union of two of its proper linear subspaces and as $L(\calV) \neq 0$, we deduce that $B_1(\calV)=\{0\}$.
As $B(\calV)\neq \{0\}$ and $\calV$ is reduced, we deduce that $B_2(\calV)=(\F_2)^{n-2}$.
Now, denote by $\alpha(M)$ the entry of $M \in \calV$ at the $(2,1)$-spot.
If $\alpha=0$, then we contradict condition (iii) in the definition of a primitive space (by deleting the second column). Thus, $\alpha \neq 0$,
$B_2 \neq 0$ and $L \neq 0$.
Fix $M \in \calV$, and note that
$$M=\begin{bmatrix}
? & ? & L(M) \\
\alpha(M) & ? & [0]_{1 \times (p-2)} \\
[0]_{(n-2) \times 1} & B_2(M) & [0]_{(n-2) \times (p-2)}
\end{bmatrix}.$$
As $\rk M \leq 2$, one of the matrices $B_2(M)$, $L(M)$ or $\alpha(M)$ must be zero.
However, the linear maps $B_2$, $L$ and $\alpha$ on $\calV$ are all non-zero, and we have just shown that
$\calV$ is the union of their respective kernels. If $p>3$, then $\Ker L$ has codimension at least $2$ in $\calV$, whence Lemma 2.5 of \cite{dSPfeweigenvalues}
yields a contradiction. Therefore, $p=3$.

By applying the above line of reasoning to $\calV^T$, we obtain $n=3$.
\end{proof}

This completes the proof of Proposition \ref{downto3}.

\subsection{Completing the classification}

Let $\calV$ be a primitive subspace of $\Mat_{n,p}(\F_2)$ with upper rank $2$. By Proposition \ref{downto3}, we know that
$\calV$ is actually a linear subspace of $\Mat_3(\F_2)$. Moreover, we can assume that $\calV$ contains $J_2$, and we keep the notation
from Section \ref{identitysection}.
We also make the following additional assumption:
\begin{itemize}
\item[(H1)] $\calV$ is inequivalent to a linear subspace of $\calJ_3(\F_2)$.
\end{itemize}
From there, our aim is to prove that $\calV$ is equivalent to $\Mata_3(\F_2)$, $\calU_3(\F_2)$ or $\calV_3(\F_2)$.

Using (H1), we see from Lemmas \ref{transrank1} and \ref{rank1} that every non-zero matrix of $\calV$ has rank $2$, and $\dim \calV x \geq 2$
 for all $x \in (\F_2)^3 \setminus \{0\}$. By Remark \ref{transposeremark}, $\calV^T$ is also inequivalent to a subspace of $\calJ_3(\F_2)$,
and hence $\dim \calV^T x \geq 2$ for all $x \in (\F_2)^3 \setminus \{0\}$.

\begin{claim}\label{alternatingclaim}
\begin{enumerate}[(a)]
\item If $D(\calV)=\{0\}$, then $\dim \calV=3$.
\item If, for every non-zero matrix $M \in \calV$, we have
$\calV \,\Ker M =\im M$, then $\calV$ is equivalent to $\Mata_3(\F_2)$.
\end{enumerate}
\end{claim}

\begin{proof}
\begin{enumerate}[(a)]
\item Assume that $D(\calV)=\{0\}$.
Denoting by $e_3$ the third vector of the standard basis of $(\F_2)^3$, we deduce that $\calV e_3 \subset (\F_2)^2 \times \{0\}$, whence $\calV e_3=(\F_2)^2 \times \{0\}$ as $\dim \calV e_3 \geq 2$.
Thus, $C(\calV)=(\F_2)^2$.
By \eqref{fundid2}, we have
$$\forall M \in \calV, \; B(M)C(M)=0.$$
Polarizing this quadratic identity yields $B(M)C(N)+B(N)C(M)=0$ for all $(M,N)\in \calV^2$.
It follows that for every $M \in \calV$ such that $C(M)=0$, we have $B(M)C(N)=0$ for all $N \in \calV$, which yields
$B(M)=0$ since $C(\calV)=(\F_2)^2$.
This yields a (non-zero) matrix $K \in \Mat_2(\F_2)$ such that $B(M)=C(M)^T K$ for all $M \in \calV$. Then,
$C(M)^T K C(M)=0$ for all $M \in \calV$, which shows that $K$ is alternating.
Therefore, $K=\begin{bmatrix}
0 & 1 \\
1 & 0
\end{bmatrix}$ (the sole non-zero matrix in $\Mata_2(\F_2)$). \\
Let $M_0 \in \calV$ be such that $C(M_0)=0$. Then, $B(M_0)=0$ and, for all $M \in \calV$, we find, by identity \eqref{fundid},
\begin{multline*}
C(M)^T K\widetilde{A(M_0)}C(M)=B(M) \widetilde{A(M_0)}C(M) \\
=B(M+M_0)\widetilde{\bigl(A(M)+A(M_0)\bigr)} C(M+M_0)-B(M) \widetilde{A(M)}C(M)=0,
\end{multline*}
whence $K\widetilde{A(M_0)}$ is alternating. Thus, $K\widetilde{A(M_0)}\in \F_2\,K$, and hence $A(M_0) \in \{0,I_2\}$.
Noting that $C(J_2)=0$, we deduce that $\Ker C=\F_2 J_2$, whence $\dim \calV=3$.

\item Assume that for every non-zero matrix $M \in \calV$, we have
$\calV \Ker M =\im M$. In particular, the case $M=J_2$ yields $D(\calV)=\{0\}$, whence the above proof shows that
$C(\calV)=(\F_2)^2$ and $B(M)=C(M)^T K$ for all $M \in \calV$.
We choose $M_1 \in \calV$ with $C(M_1)=\begin{bmatrix}
1 \\
0
\end{bmatrix}$, so that
$$M_1=\begin{bmatrix}
? & ? & 1 \\
? & ? & 0 \\
0 & 1 & 0
\end{bmatrix}.$$
Replacing $M_1$ with $M_1+J_2$ if necessary, we can assume that the entry of $M_1$
at the $(1,1)$-spot is $0$. As $M_1$ is singular, its entry at the $(2,1)$-spot is $0$.
Using row operations of the form $L_1 \leftarrow L_1-\lambda L_3$ and $L_2 \leftarrow L_2-\mu L_3$,
we see that no generality is lost in assuming that
$$M_1=\begin{bmatrix}
0 & 0 & 1 \\
0 & 0 & 0 \\
0 & 1 & 0
\end{bmatrix}.$$
With a similar line of reasoning, we find scalars $a$ and $b$ such that $\calV$ contains a matrix of the form
$$M_2=\begin{bmatrix}
0 & 0 & 0 \\
a & b & 1 \\
1 & 0 & 0
\end{bmatrix}.$$
As $\rk(M_1+M_2) \leq 2$, one finds $a=b$ by computing the determinant.
As $e_1 \in \Ker M_1$, we must have $M_2 e_1 \in \im M_1$, whence $a=0$. We conclude that $a=b=0$,
and hence, as $\dim \calV=3$, we have $\calV=\Vect(J_2,M_1,M_2)$, i.e.\ $\calV$ is associated with the generic matrix
$$\begin{bmatrix}
\mathbf{a} & 0 & \mathbf{b} \\
0 & \mathbf{a} & \mathbf{c} \\
\mathbf{c} & \mathbf{b} & 0
\end{bmatrix}.$$
Swapping the first two rows finally shows that $\calV$ is equivalent to $\Mata_3(\F_2)$.
\end{enumerate}
\end{proof}

\begin{claim}\label{transitivityclaim}
One has $3 \leq \dim \calV \leq 4$ and there is at least one non-zero vector $x \in (\F_2)^3$ for which $\dim \calV x=2$.
Moreover, if $\dim \calV=3$, then $\dim \calV x=\dim \calV^T x=2$ for all non-zero vectors $x \in (\F_2)^3$.
\end{claim}

\begin{proof}
Set $d:=\dim \calV$.
We use a counting argument: Denote by $\calN$ the set of all pairs  $(M,x)\in \bigl(\calV \setminus \{0\}\bigr) \times \bigl((\F_2)^3 \setminus \{0\}\bigr)$
for which $Mx=0$. Remember that $\dim \calV x \in \{2,3\}$ for all non-zero vectors $x \in (\F_2)^3$.
For $i \in \{2,3\}$, denote by
$n_i$ the number of non-zero vectors $x \in (\F_2)^3$ for which $\dim \calV x=i$.
For every non-zero vector $x \in (\F_2)^3$, the set of all matrices $M \in \calV$ for which $Mx=0$ is the kernel of
$M \mapsto Mx$, and hence it has dimension $d-\dim \calV x$. Thus,
$$\# \calN=(2^{d-2}-1)\,n_2+(2^{d-3}-1)\,n_3.$$
On the other hand, every non-zero matrix of $\calV$ has rank $2$ and hence it annihilates exactly one non-zero vector of $(\F_2)^3$.
Therefore,
$$\# \calN= 2^d-1.$$
As $n_3=7-n_2$, we deduce that $2^{d-3} n_2=2^d- 7\times 2^{d-3}+6$, which leads to
$$n_2=1+3\times 2^{4-d}.$$
In particular, we deduce that $n_2>0$.
As $n_2$ must be an integer, we find $4-d \geq 0$.
As $n_2 \leq 7$, we also find $d \geq 3$. Thus, $d \in \{3,4\}$.
Finally, if $d=3$, then $n_2=7$ whence $\dim \calV x=2$ for every non-zero vector $x \in (\F_2)^3$;
$\calV^T$ must satisfy the same conclusion as it has dimension $3$.
\end{proof}

Now, we make an additional assumption:

\begin{itemize}
\item[(H2)] $\calV$ is inequivalent to $\Mata_3(\F_2)$.
\end{itemize}

We shall conclude by distinguishing between two cases, whether $\calV$ has dimension $3$ or $4$.

\begin{claim}
Assume that $\dim \calV=3$. Then, $\calV$ is equivalent to $\calU_3(\F_2)$.
\end{claim}

\begin{proof}
If there are two distinct matrices of $\calV$ with the same (two-dimensional) range,
then, by choosing a non-zero vector $x$ in the orthogonal complement of this range, we would find
$\dim \calV^T x \leq 1$, contradicting Claim \ref{transitivityclaim}.
Thus, two distinct matrices of $\calV$ cannot have the same range.

Combining point (b) of Claim \ref{alternatingclaim} with assumption (H2), we find a matrix $M \in \calV$
such that $\calV \Ker M \neq \im M$.
We lose no generality in assuming that $M=J_2$. As $\dim \calV e_3=2$, no further generality is lost in assuming that the space of all third columns of the matrices of $\calV$ is $\F_2 \times \{0\} \times \F_2$.
This yields two matrices in $\calV$ of the following forms
$$M_1=\begin{bmatrix}
? & ? & 1 \\
? & 0 & 0 \\
? & ? & 0
\end{bmatrix} \quad \text{and} \quad
M_2=\begin{bmatrix}
0 & ? & 0 \\
? & ? & 0 \\
? & ? & 1
\end{bmatrix},$$
as we may add $J_2$ if necessary.
Note that $J_2,M_1,M_2$ are obviously linearly independent whence $\calV=\Vect(J_2,M_1,M_2)$.
By identity \eqref{fundid2}, the entry of $M_1$ at the $(3,1)$-spot must be zero.
As $M_1 \neq J_2$ and $\rk M_1=\rk J_2=2$, we must have $\im M_1 \not\subset \im J_2$, and hence
the entry of $M_1$ at the $(3,2)$-spot is non-zero.
It follows that
$$M_1=\begin{bmatrix}
? & ? & 1 \\
? & 0 & 0 \\
0 & 1 & 0
\end{bmatrix}.$$
As $\rk M_1=2$, we deduce that
$$M_1=\begin{bmatrix}
? & ? & 1 \\
0 & 0 & 0 \\
0 & 1 & 0
\end{bmatrix}.$$
Performing column operations of the forms $C_1 \leftarrow C_1-\lambda C_3$ and $C_2 \leftarrow C_2-\mu C_3$,
we see that no generality is lost in assuming that
$$M_1=\begin{bmatrix}
0 & 0 & 1 \\
0 & 0 & 0 \\
0 & 1 & 0
\end{bmatrix}.$$
As $\dim \calV e_2=2$ and $\calV$ contains $J_2$ and $M_1$, the entry of $M_2$ at the $(1,2)$-spot is zero, whence
$$M_2=\begin{bmatrix}
0 & 0 & 0 \\
? & ? & 0 \\
? & ? & 1
\end{bmatrix}.$$
As $J_2+M_2$ is singular, we deduce that
$$M_2=\begin{bmatrix}
0 & 0 & 0 \\
a & 1 & 0 \\
b & c & 1
\end{bmatrix} \quad \text{for some $(a,b,c)\in (\F_2)^3$.}$$
As $M_1+M_2$ is singular, we find $a(c+1)=b$ by computing the determinant.
Using the singularity of $M_1+M_2+J_2$, we obtain $a(c+1)=0$. Thus, $b=0$.
However, as $\dim(\calV e_1)=2$ and $\calV=\Vect(J_2,M_1,M_2)$, we cannot have $a=0$, whence $a=1$ and $c=1$.
Thus,
$$M_2=\begin{bmatrix}
0 & 0 & 0 \\
1 & 1 & 0 \\
0 & 1 & 1
\end{bmatrix}.$$
We deduce that $\calV$ is associated with the generic matrix
$$\begin{bmatrix}
\mathbf{a} & 0 & \mathbf{b} \\
\mathbf{c} &  \mathbf{a}+\mathbf{c} & 0 \\
0 & \mathbf{b}+\mathbf{c} & \mathbf{c}
\end{bmatrix}.$$
Using the operations $L_1 \leftrightarrow L_3$ and $C_2 \leftrightarrow C_3$, we obtain that $\calV$ is equivalent to $\calU_3(\F_2)$.
\end{proof}

\begin{claim}
Assume that $\dim \calV=4$. Then, $\calV$ is equivalent to $\calV_3(\F_2)$.
\end{claim}

\begin{proof}
By Claim \ref{transitivityclaim}, we can choose a vector $x \in (\F_2)^3$ for which $\dim \calV x=2$.
Then, there is a non-zero matrix of $\calV$ which annihilates $x$, and this matrix has rank $2$.
Without loss of generality, we may assume that this matrix is $J_2$, in which case $x$ is the third vector of the standard
basis. Point (a) of Claim \ref{alternatingclaim} yields that $\calV x \neq \im J_2$;
Thus, we can choose a basis $(y,y')$ of $\calV x$ such that $y \in \im J_2$ and $y' \not\in \im J_2$,
then we choose $x_1 \in (\F_2)^3$ such that $J_2 x_1=y$, and we extend $(x_1,x)$ into a basis $(x_1,x_2,x)$ of $(\F_2)^3$. 
Then, by replacing $\calV$ with an equivalent subspace -- so that our new source basis is $(x_1,x_2,x)$ and our new target basis is $(y,J_2x_2,y')$ --
we see that no further generality is lost in assuming that $\calV x=\F_2 \times \{0\} \times \F_2$.
Since $\dim \calV=4$, we may extend $J_2$ into a basis $(J_2,M_1,M_2,M_3)$ of
$\calV$ with
$$M_1=\begin{bmatrix}
? & ? & 0 \\
? & ? & 0 \\
? & ? & 0
\end{bmatrix}, \; M_2=\begin{bmatrix}
? & ? & 1 \\
? & ? & 0 \\
? & ? & 0
\end{bmatrix} \quad \text{and} \quad
M_3=\begin{bmatrix}
? & ? & 0 \\
? & ? & 0 \\
? & ? & 1
\end{bmatrix}.$$
Adding $J_2$ to $M_1$ and $M_2$ if necessary, we may assume that
$$M_1=\begin{bmatrix}
0 & ? & 0 \\
? & ? & 0 \\
? & ? & 0
\end{bmatrix} \quad \text{and} \quad  M_2=\begin{bmatrix}
? & ? & 1 \\
? & 0 & 0 \\
? & ? & 0
\end{bmatrix}.$$
Applying identity \eqref{fundid2} to $M_2$ and $M_1+M_2$, we find that
$$M_1=\begin{bmatrix}
0 & ? & 0 \\
? & ? & 0 \\
0 & ? & 0
\end{bmatrix}, \; M_2=\begin{bmatrix}
? & ? & 1 \\
? & 0 & 0 \\
0 & ? & 0
\end{bmatrix}.$$
Assume first that $B(M_1)=0$.
As $\dim(\calV^T e_3) \geq 2$, we must have $B(M_2)\neq 0$, whence
$B(M_2)=\begin{bmatrix}
0 & 1
\end{bmatrix}$.
Then, as $M_2$ is singular, we find that $(M_2)_{2,1}=0$ and, as
$M_1+M_2$ is singular, we also obtain $(M_1)_{2,1}=0$. Thus, the first and third columns of $M_1$
equal zero, whence $M_1$ has rank $1$, which is absurd.

We deduce that $B(M_1)=\begin{bmatrix}
0 & 1
\end{bmatrix}$. Then, as we lose no generality in replacing $M_2$ with a matrix of the form $M_2+a\,M_1+b\,J_2$, we can assume that $B(M_2)=0$.
From there, using column operations of the form $C_1 \leftarrow C_1+\lambda C_3$ and $C_2 \leftarrow C_2+\mu C_3$,
we see that no generality is lost in assuming that
$$M_2=\begin{bmatrix}
0 & 0 & 1 \\
? & 0 & 0 \\
0 & 0 & 0
\end{bmatrix}.$$
As $\rk M_2>1$, we deduce that
$$M_2=\begin{bmatrix}
0 & 0 & 1 \\
1 & 0 & 0 \\
0 & 0 & 0
\end{bmatrix}.$$
If the entry of $M_1$ at the $(1,2)$-spot equals $1$, then we perform the row operation
$L_1 \leftarrow L_1+L_3$ and then we replace $M_3$ with $M_3+M_2$.
This shows that no generality is lost  in assuming that
$$M_1=\begin{bmatrix}
0 & 0 & 0 \\
? & ? & 0 \\
0 & 1 & 0
\end{bmatrix}.$$
As $M_1$ has rank $2$, we find
$$M_1=\begin{bmatrix}
0 & 0 & 0 \\
1 & a & 0 \\
0 & 1 & 0
\end{bmatrix} \quad \text{for some $a \in \F_2$.}$$
From there, we lose no generality in adding a linear combination of
$M_1$ and $J_2$ to $M_3$, whence we may assume that
$$M_3=\begin{bmatrix}
0 & c & 0 \\
0 & d & 0 \\
b & e & 1
\end{bmatrix} \quad \text{for some $(b,c,d,e)\in (\F_2)^4$.}$$
As $\det(J_2+M_3)=0$, $\det(M_1+M_3)=0$ and $\det(J_2+M_1+M_3)=0$,
we find $d=1$, $c=0$ and $a=0$, successively.
Finally, using $\det(M_2+M_3)=0$ and $\det(J_2+M_2+M_3)=0$, we find $b=e$ and $e=0$.
Thus, $\calV$ is the span of the matrices
$$J_2=\begin{bmatrix}
1 & 0 & 0 \\
0 & 1 & 0 \\
0 & 0 & 0
\end{bmatrix}, \; M_1=\begin{bmatrix}
0 & 0 & 0 \\
1 & 0 & 0 \\
0 & 1 & 0
\end{bmatrix}, \; M_2=\begin{bmatrix}
0 & 0 & 1 \\
1 & 0 & 0 \\
0 & 0 & 0
\end{bmatrix} \quad \text{and} \quad M_3=\begin{bmatrix}
0 & 0 & 0 \\
0 & 1 & 0 \\
0 & 0 & 1
\end{bmatrix},$$
whence it is associated with the generic matrix
$$\begin{bmatrix}
\mathbf{a} & 0 & \mathbf{c} \\
\mathbf{b}+\mathbf{c} & \mathbf{a}+\mathbf{d} & 0 \\
0 & \mathbf{b} & \mathbf{d}
\end{bmatrix}.$$
Using the column operations $C_2 \leftrightarrow C_1$ and $C_3 \leftrightarrow C_2$, we conclude that $\calV$
is equivalent to $\calV_3(\F_2)$, as claimed.
\end{proof}

This completes the proof of Theorem \ref{classrank2F2}.

\subsection{Application to maximal spaces of matrices with upper rank $2$}

\begin{Not}
Given integers $n' \in \lcro 0,n\rcro$ and $p' \in \lcro 0,p\rcro$ together with a subspace $\calW$ of $\Mat_{n',p'}(\K)$, we
denote by $\widetilde{\calW}^{(n,p)}$ the space of all $n \times p$ matrices of the form
$$\begin{bmatrix}
M & [0]_{n'  \times (p-p')} \\
[0]_{(n-n') \times p'} & [0]_{(n-n') \times (p-p')}
\end{bmatrix} \quad \text{with $M \in \calW$.}$$
Note that $\widetilde{\calW}^{(n,p)}$ has the same upper rank as $\calW$.
\end{Not}

Let $\calS$ be a linear subspace of $\calL(U,V)$, where $U$ and $V$ are finite-dimensional vector spaces.
We define the kernel and the range of $\calS$ as, respectively, $\Ker S:=\underset{f \in \calS}{\bigcap} \Ker f$
and $\im \calS:=\underset{f \in \calS}{\sum} \im f$. Then, every operator $f \in \calS$ induces a linear operator
$\overline{f} : U/\Ker \calS \rightarrow \im \calS$, and one sees that the operator space $\overline{\calS} :=\{\overline{f} \mid f \in \calS\}$
is a reduced linear subspace with the same dimension and the same upper rank as $\calS$: It is called the \textbf{reduced operator space} of $\calS$. Finally, two operator subspaces
$\calS$ and $\calT$ of $\calL(U,V)$ are equivalent if and only if $\dim \Ker \calS=\dim \Ker \calT$,
$\dim \im \calS=\dim \im \calT$ and the operator spaces $\overline{\calS}$ and $\overline{\calT}$ are equivalent.

In terms of matrices, this reads as follows:

\begin{prop}\label{uniquenessprop}
Let $\calV$ be a linear subspace of $\Mat_{n,p}(\K)$ with upper rank $r$.
Then, there is a pair $(n',p')\in \lcro 0,n\rcro \times \lcro 0,p\rcro$ and a reduced linear subspace $\calV'$ of
$\Mat_{n',p'}(\K)$ such that $\calV \sim \widetilde{\calV'}^{(n,p)}$.
The pair $(n',p')$ is uniquely determined by $\calV$, and the equivalence class of $\calV'$ is uniquely determined by that of $\calV$.
\end{prop}

Note that $\calV'$ represents the reduced operator space of $\calV$ (seen as a space of linear maps from $\K^p$ to $\K^n$).
If $\calV$ is equivalent to a subspace of $\calR(2,0)$ (respectively, of $\calR(0,2)$), then $n' \leq 2$ (respectively, $p' \leq 2$).
Moreover, if $\calV'$ is equivalent to a subspace of $\calR(1,1)$, then $\calV$ is equivalent to a subspace of $\calR(1,1)$.
Conversely, assuming that $\calV \subset \calR(1,1)$, then we have a hyperplane $H$ of $\K^p$ and a $1$-dimensional subspace $D$ of $\K^n$ such that
$\calV x \subset D$ for all $x \in H$. Then, for $H':=(H+\Ker \calV)/\Ker \calV$ and $D':=D \cap \im \calV$,
we see that $\overline{f}(x) \in D$ for all $x \in H'$ and all $f \in \calV$, and $H'$ has codimension at most $1$ in $\K^p/\Ker \calV$ whereas
$D'$ has dimension at most $1$. It follows that $\calV'$ is equivalent to a subspace of $\calR(1,1)$.
From the above considerations combined with Proposition \ref{basicsemiprimitiveprop}
and Theorem \ref{classrank2F2}, we deduce the following structure theorem on subspaces of matrices of $\Mat_3(\F_2)$ with rank at most $2$:

\begin{theo}\label{allrank2spaces}
Let $\calV$ be an upper rank $2$ subspace of $\Mat_{n,p}(\F_2)$, with $n \geq 3$ and $p \geq 3$.
Then, one and only one of the following cases holds:
\begin{enumerate}[(i)]
\item $\calV$ is equivalent to a subspace of $\calR(2,0)$;
\item $\calV$ is equivalent to a subspace of $\calR(0,2)$;
\item $\calV$ is equivalent to a subspace of $\calR(1,1)$;
\item $\calV$ is equivalent to $\widetilde{\calW}^{(n,p)}$, where $\calW$ is a primitive linear subspace of $\calJ_3(\F_2)$;
\item $\calV$ is equivalent to $\widetilde{\Mata_3(\F_2)}^{(n,p)}$;
\item $\calV$ is equivalent to $\widetilde{\calU_3(\F_2)}^{(n,p)}$;
\item $\calV$ is equivalent to $\widetilde{\calV_3(\F_2)}^{(n,p)}$.
\end{enumerate}
Moreover, if case (iv) holds, then the equivalence class of $\calW$ is uniquely determined by that of $\calV$.

\end{theo}

As a consequence, we get:

\begin{theo}[Classification of maximal spaces of matrices with rank at most $2$]
Let $n>2$ and $p>2$.
Up to equivalence, there are $6$ maximal subspaces of upper rank $2$ matrices of $\Mat_{n,p}(\F_2)$:
$$\calR(2,0), \quad \calR(0,2), \quad \calR(1,1), \quad \widetilde{\calJ_3(\F_2)}^{(n,p)}, \quad \widetilde{\Mata_3(\F_2)}^{(n,p)} \quad \text{and} \quad
\widetilde{\calV_3(\F_2)}^{(n,p)}.$$
\end{theo}

The only non-trivial point in the derivation of that theorem from Theorem \ref{allrank2spaces} and Proposition \ref{uniquenessprop}
is to see that $\Mata_3(\F_2)$ is maximal among the subspaces of $\Mat_3(\F_2)$ with upper rank $2$.
This is obtained as a special case of the following general result:

\begin{prop}
Let $n$ be an odd integer and $\F$ be an arbitrary field. Then, $\Mata_n(\F)$ is a maximal subspace of singular matrices of $\Mat_n(\F)$.
\end{prop}

\begin{proof}
For the case when $\# \F>2$, we refer to \cite[Proposition 5]{FillmoreLaurieRadjavi}.
Thus, we shall only consider the case when $\F=\F_2$. We note that the problem is tightly connected to the representation of quadratic forms.
We refer to \cite[Chapter XXXII]{invitquad} for the basics on quadratic forms over fields of characteristic $2$.
Let $P \in \Mat_n(\F_2) \setminus \Mata_n(\F_2)$. We have to show that $P+\Mata_n(\F_2)$ contains a non-singular matrix.
We consider the non-zero quadratic form $q : X \in \F_2^n \mapsto X^T PX$.
The set of matrices $Q \in \Mat_n(\F_2)$ that represent $q$, i.e.\ such that, in some basis of $(\F_2)^n$,
the map $X \mapsto X^T Q X$ corresponds to $q$, is precisely $\Co(P)+\Mata_n(\F_2)$, where $\Co(P)$ denotes the congruence class of $P$,
that is the set of all matrices $RQR^T$ with $R \in \GL_n(\F_2)$. As $\Mata_n(\F_2)$ is invariant under congruence,
it suffices to find a non-singular matrix which represents $q$.
The rank of the polar form of $q$ equals $2r$ for some non-negative integer $r$.
As $n$ is odd, the radical of $q$ is odd-dimensional, and hence non-zero.
The restriction of $q$ to its radical is a linear form.
We shall now distinguish between two cases, whether this linear form is zero or not.
For $(a,b)\in (\F_2)^2$, we denote by $[a,b]$ the quadratic form $(x,y) \mapsto ax^2+xy+by^2$ on $(\F_2)^2$, and by
$\langle a\rangle$ the quadratic form $x \mapsto ax^2$. The orthogonal direct sum of two quadratic forms $q_1$ and $q_2$
is denoted by $q_1 \bot q_2$.

\noindent \textbf{Case 1.} The restriction of $q$ to its radical $R$ is non-zero. \\
Then, we may choose a basis of $R$ in which no vector is $q$-isotropic
(indeed, the set of $q$-isotropic vectors in $R$ is a linear hyperplane of $R$, and hence its complementary set
in $R$ spans $R$).
This yields pairs $(a_1,b_1),\dots,(a_r,b_r)$ in $(\F_2)^2$ such that
$q$ is equivalent to $[a_1,b_1] \bot \cdots \bot [a_r,b_r] \bot \langle 1\rangle \bot \cdots \bot \langle 1 \rangle$.
However, the contamination lemma \cite[Chapter XXXII, Lemma 5.4.2]{invitquad} shows that, for all $(a,b)\in (\F_2)^2$,
$$[a,b] \bot \langle 1\rangle \simeq [a+1,b] \bot \langle 1\rangle \simeq [a,b+1] \bot \langle 1\rangle
\simeq  [a+1,b+1] \bot \langle 1\rangle.$$
Using this repeatedly, we deduce that $q$ is equivalent to $r.[1,1] \bot (n-2r).\langle 1\rangle$,
 whence the invertible matrix
 $\begin{bmatrix}
 I_r & I_r \\
 0 & I_r
 \end{bmatrix} \oplus I_{n-2r}$ represents $q$.

\noindent \textbf{Case 2.} The restriction of $q$ to its radical is zero. \\
Then, $r \geq 1$ as $q$ is non-zero. Using
 the equivalence $[1,1] \bot [1,1] \simeq [0,0] \bot [0,0]$ (see \cite[Chapter XXXII, Example 4.2.3]{invitquad}),
 we find that $q$ is equivalent to $(r-1).[1,1]\bot (n-2r-1).\langle 0\rangle \bot \varphi$,
 where $\varphi$ equals either $[0,0] \bot \langle 0\rangle$ or $[1,1] \bot \langle 0\rangle$.
 As $n$ is odd, we have $n-2r-1=2s$ for some non-negative integer $s$, whence
  $(r-1).[1,1]\bot (n-2r-1).\langle 0\rangle$ is represented by the non-singular matrix
  $\begin{bmatrix}
  I_{r-1} & I_{r-1}\\
  0 & I_{r-1}
  \end{bmatrix} \oplus \begin{bmatrix}
  0 & I_s\\
  I_s & 0
  \end{bmatrix}$. Therefore, it only remains to prove that $\varphi$ is represented by at least one non-singular matrix.
 If $\varphi$ equals $[0,0] \bot \langle 0\rangle$, then it is represented by the non-singular matrix $\begin{bmatrix}
 0 & 1 & 1 \\
 0 & 0 & 1 \\
 1 & 1 & 0
 \end{bmatrix}$. If $\varphi$ equals $[1,1] \bot \langle 0\rangle$, then it is represented by the non-singular matrix
  $\begin{bmatrix}
 1 & 1 & 0 \\
 0 & 1 & 1 \\
 0 & 1 & 0
 \end{bmatrix}$. In any case, the conclusion follows that $q$ is represented by at least one non-singular matrix.
\end{proof}

\section{Primitive linear subspaces of $\calJ_3(\F_2)$}\label{structureofJ3section}

\subsection{A rough result on the primitive subspaces of $\calJ_3(\F_2)$}

Let us prove Proposition \ref{J3structure}.
Let $\calV$ be a linear subspace of $\calJ_3(\F_2)$, and consider the space $\calD \subset (\F_2)^3$
of all diagonal vectors in $\calV$, that is the space of all vectors $\begin{bmatrix}
m_{1,1} & m_{2,2} & m_{3,3}
\end{bmatrix}^T$ with $M \in \calV$.
We know that $\calD$ is included in the hyperplane $\calH=\{(a,b,c) \in (\F_2)^3 : \; a+b+c=0\}$.

If $\calD=\{0\}$, then it is obvious that $\calV$ is not reduced, whence it is not primitive. \\
Assume now that $\dim \calD=1$. Then, $\calD$ contains a sole non-zero vector which we write $\begin{bmatrix}
a & b & c
\end{bmatrix}^T$. As $\calD$ is reduced, we must have $a=1$ and $c=1$. Thus, $b=0$ and, by swapping the first and third columns,
we see that $\calV$ is equivalent to a linear subspace of $\calR(1,1)$.

Now, we assume that $\calD=\calH$ and we prove that $\calV$ is primitive.
Note that $\calV^T$ is equivalent to a linear subspace of $\calJ_3(\F_2)$ for which the space of all diagonal vectors is $\calH$.
Now, let $X=\begin{bmatrix}
x & y & z
\end{bmatrix}^T \in (\F_2)^3$ be such that $MX=0$ for all $M \in \calV$. Looking at the third entry of the matrix $MX$,
we deduce that $z=0$. Then, we successively find $y=0$ and $x=0$ by looking at the second entry of $MX$ and then at the first one.
Thus, $\calV$ satisfies condition (i) in the definition of a primitive space. For the same reason $\calV^T$
also does, whence $\calV$ is reduced.

Denote by $(e_1,e_2,e_3)$ the standard basis of $(\F_2)^3$. Let $x \in (\F_2)^3 \setminus \F_2 e_1$.
We contend that $\calV x \neq \F_2 e_1$. Indeed,
if the third entry of $x$ equals $1$ then, as we know that some matrix of $\calV$ has entry $1$ at the $(3,3)$-spot,
we see that $\calV x \neq \F_2 e_1$; Otherwise, the second entry of $x$ equals $1$ and as some matrix of  $\calV$ has entry $1$ at the $(2,2)$-spot
we obtain that $\calV x \neq \F_2 e_1$.

Now, assume that $\calV$ is non-primitive. Then, as it is reduced, it must be equivalent to a subspace of $\calR(1,1)$, which
yields a $2$-dimensional subspace $P$ of $(\F_2)^3$ and a $1$-dimensional subspace $D$ of $(\F_2)^3$ such that
$\calV x \subset D$ for all $x \in P$.
As $P \not\subset \F_2 e_1$, the above proof yields that $D \neq \F_2 e_1$
whence $e_1 \not\in P$. Therefore, $(\F_2)^3=\F_2 e_1\oplus P$, which yields
$\calV x \in D+\F_2 e_1$ for all $x \in (\F_2)^3$, contradicting the fact that $\calV$ is reduced.

We conclude that $\calV$ is primitive, which finishes the proof of Proposition \ref{J3structure}.

\subsection{The full classification of primitive subspaces of $\calJ_3(\F_2)$}

Now, we shall give a full classification, up to equivalence, of the primitive
subspaces of $\calJ_3(\F_2)$. Of course, we have just seen that $\calJ_3(\F_2)$ is primitive, whence
it only remains to classify its primitive subspaces with dimension $2$, $3$ or $4$
(obviously, a subspace of $\Mat_3(\F_2)$ with upper rank $2$ and dimension at most $1$ is non-reduced).
This is given in the next three propositions:

\begin{prop}\label{J3dim2}
Let $\calV$ be a primitive subspace of $\calJ_3(\F_2)$ with dimension $2$. Then,
$\calV$ is equivalent to the space associated with the generic matrix
$$\begin{bmatrix}
\mathbf{a} & 0 & 0 \\
0 & \mathbf{a}+\mathbf{b} & 0 \\
0 & 0 & \mathbf{b}
\end{bmatrix}.$$
\end{prop}

\begin{prop}\label{J3dim3}
Let $\calV$ be a primitive subspace of $\calJ_3(\F_2)$ with dimension $3$. Then,
$\calV$ is equivalent to one and only one of the four spaces associated with the generic matrices
$$\mathbf{M}_1:=\begin{bmatrix}
\mathbf{a} & 0 & \mathbf{c}  \\
0 & \mathbf{a}+\mathbf{b} & 0 \\
0 & 0 & \mathbf{b}
\end{bmatrix}, \quad
\mathbf{M}_2:=\begin{bmatrix}
\mathbf{a} & \mathbf{c} & 0 \\
0 & \mathbf{a}+\mathbf{b} & \mathbf{a} \\
0 & 0 & \mathbf{b}
\end{bmatrix},$$
$$\mathbf{M}_3:=\begin{bmatrix}
\mathbf{a} & \mathbf{b} & 0 \\
0 & \mathbf{a}+\mathbf{b} & \mathbf{c} \\
0 & 0 & \mathbf{b}
\end{bmatrix} \quad \text{and} \quad
\mathbf{M}_4:=\begin{bmatrix}
\mathbf{a} & \mathbf{c} & 0 \\
0 & \mathbf{a}+\mathbf{b} & \mathbf{c} \\
0 & 0 & \mathbf{b}
\end{bmatrix}.$$
\end{prop}

\begin{prop}\label{J3dim4}
Let $\calV$ be a primitive subspace of $\calJ_3(\F_2)$ with dimension $4$. Then,
$\calV$ is equivalent to one and only one of the four spaces associated with the generic matrices
$$\mathbf{N}_1:=\begin{bmatrix}
\mathbf{a} & \mathbf{c} & 0 \\
0 & \mathbf{a}+\mathbf{b} & \mathbf{d} \\
0 & 0 & \mathbf{b}
\end{bmatrix}, \quad
\mathbf{N}_2:=\begin{bmatrix}
\mathbf{a} & \mathbf{c} & \mathbf{d} \\
0 & \mathbf{a}+\mathbf{b} & 0 \\
0 & 0 & \mathbf{b}
\end{bmatrix},$$
$$\mathbf{N}_3:=\begin{bmatrix}
\mathbf{a} & 0 & \mathbf{c} \\
0 & \mathbf{a}+\mathbf{b} & \mathbf{d} \\
0 & 0 & \mathbf{b}
\end{bmatrix} \quad \text{and} \quad
\mathbf{N}_4:=\begin{bmatrix}
\mathbf{a} & \mathbf{c} & \mathbf{d} \\
0 & \mathbf{a}+\mathbf{b} & \mathbf{c} \\
0 & 0 & \mathbf{b}
\end{bmatrix}.$$
\end{prop}

\begin{Rem}\label{rank1J3remark}
In the prospect of the proofs of Propositions \ref{J3dim3} and \ref{J3dim4}, the following remark will be useful:
the set of all matrices with rank at most $1$ in $\calJ_3(\F_2)$ is the union of the $2$-dimensional subspaces
\[P_1:=\Biggl\{\begin{bmatrix}
0 & a & b \\
0 & 0 & 0 \\
0 & 0 & 0
\end{bmatrix}\mid (a,b)\in (\F_2)^2\Biggr\}\quad \text{and} \quad
P_2:=\Biggl\{\begin{bmatrix}
0 & 0 & b \\
0 & 0 & a \\
0 & 0 & 0
\end{bmatrix}\mid (a,b)\in (\F_2)^2\Biggr\}.\]
Indeed, if a matrix $M\in \calJ_3(\F_2)$ has its diagonal non-zero, then exactly two of its diagonal entries equal $1$, whence it has rank $2$.
Thus, a rank $1$ matrix of $\calJ_3(\F_2)$ must have its diagonal zero: From there, the claimed result is obvious.
\end{Rem}

\begin{proof}[Proof of Proposition \ref{J3dim3}]
Using Proposition \ref{J3structure} together with the rank theorem, we see that $\calV$ contains exactly one non-zero matrix $M_0$ with diagonal zero.
We split the discussion into four cases, according to the value of $M_0$.

\noindent \textbf{Case 1.} $M_0=\begin{bmatrix}
0 & 0 & 1 \\
0 & 0 & 0 \\
0 & 0 & 0
\end{bmatrix}$. \\
Then, we find scalars $\alpha,\beta,\gamma,\delta$ such that a generic matrix of $\calV$ is
\[\begin{bmatrix}
\mathbf{a} & \alpha \mathbf{a}+\beta \mathbf{b} & \mathbf{c} \\
0 & \mathbf{a}+\mathbf{b} & \gamma \mathbf{a}+\delta \mathbf{b} \\
0 & 0 & \mathbf{b}
\end{bmatrix}.\]
Performing the operations $C_3 \leftarrow C_3+\gamma C_2$, $L_2 \leftarrow L_2+(\delta+\gamma)L_3$,
$L_1 \leftarrow L_1+\beta L_2$ and $C_2 \leftarrow C_2+(\alpha+\beta) C_1$, we reduce the situation to the one where
$\alpha=\beta=\gamma=\delta=0$, and hence $\calV$ is equivalent to the space associated with $\mathbf{M}_1$.

\noindent \textbf{Case 2.} $M_0=\begin{bmatrix}
0 & 1 & ? \\
0 & 0 & 0 \\
0 & 0 & 0
\end{bmatrix}$. \\
Then, by performing the column operation $C_3 \leftarrow C_3+C_2$ if necessary, we see that no generality is lost in assuming that
$M_0=\begin{bmatrix}
0 & 1 & 0 \\
0 & 0 & 0 \\
0 & 0 & 0
\end{bmatrix}$, whence we find scalars $\alpha,\beta,\gamma,\delta$ such that a generic matrix of $\calV$ is
\[\begin{bmatrix}
\mathbf{a} & \mathbf{c} & \gamma \mathbf{a}+\delta \mathbf{b} \\
0 & \mathbf{a}+\mathbf{b} & \alpha \mathbf{a}+\beta \mathbf{b} \\
0 & 0 & \mathbf{b}
\end{bmatrix}.\]
Using the operations $C_3 \leftarrow C_3+\gamma C_1$, $L_1 \leftarrow L_1+\delta L_3$ and $L_2 \leftarrow L_2+\beta L_3$,
we reduce the situation to the one where $\gamma=\delta=\beta=0$.
If $\alpha=1$, then $\calV$ is equivalent to the matrix space associated with $\mathbf{M}_2$.
If $\alpha=0$  then permuting rows and columns shows that $\calV$ is equivalent to the matrix space associated with $\mathbf{M}_1$.

\noindent \textbf{Case 3.} $M_0=\begin{bmatrix}
0 & 0 & ? \\
0 & 0 & 1 \\
0 & 0 & 0
\end{bmatrix}$. \\
With a similar line of reasoning as in Case 2, one finds that $\calV$ is equivalent to the matrix space associated with $\mathbf{M}_3$
or to the one associated with $\begin{bmatrix}
\mathbf{a} & 0 & 0 \\
0 & \mathbf{a}+\mathbf{b} & \mathbf{c} \\
0 & 0 & \mathbf{b}
\end{bmatrix}$, which is easily seen to be equivalent to the one associated with $\mathbf{M}_1$.

\noindent \textbf{Case 4.} $M_0=\begin{bmatrix}
0 & 1 & ? \\
0 & 0 & 1 \\
0 & 0 & 0
\end{bmatrix}$. \\
Using $C_3 \leftarrow C_3+C_2$ if necessary, we see that no generality is lost in assuming that
$M_0=\begin{bmatrix}
0 & 1 & 0 \\
0 & 0 & 1 \\
0 & 0 & 0
\end{bmatrix}$. Then, we have scalars $\alpha,\beta,\gamma,\delta,\eta,\epsilon$ such that $\calV$ is associated with the generic matrix
\[\begin{bmatrix}
\mathbf{a} & \alpha \mathbf{a}+\beta \mathbf{b}+\mathbf{c} & \eta \mathbf{a}+\epsilon \mathbf{b} \\
0 & \mathbf{a}+\mathbf{b} & \gamma \mathbf{a}+\delta \mathbf{b}+\mathbf{c} \\
0 & 0 & \mathbf{b}
\end{bmatrix}.\]
Using the operations $C_2 \leftarrow C_2+(\alpha+\gamma) C_1$, $L_2 \leftarrow L_2+(\beta+\delta) L_3$,
$L_1 \leftarrow L_1+\epsilon L_3$ and
$C_3 \leftarrow C_3+\eta C_1$, we finally reduce the situation to the one where $\calV$ is associated with the generic matrix $\mathbf{M}_4$.

It remains to show that the four cited matrix spaces are pairwise inequivalent.
To do this, we note that the equivalence class of a matrix subspace $\calW$ of $\Mat_3(\F_2)$ determines
both the number of vectors $x \in (\F_2)^3$ for which $\dim (\calW x)=1$ and the number of vectors
$x \in (\F_2)^3$ for which $\dim (\calW^T x)=1$. For the above four matrix spaces, we obtain the following results, which
show that they are pairwise inequivalent:

\begin{center}
\begin{tabular}{| c || c | c | c | c |}
\hline
$\calV$ associated with the generic matrix \ldots & $\mathbf{M}_1$ & $\mathbf{M}_2$ & $\mathbf{M}_3$ & $\mathbf{M}_4$ \\
\hline
\hline
Number of vectors $x \in (\F_2)^3$ such that $\dim (\calV x)=1$ & $2$ & $1$ & $2$ & $1$ \\
\hline
Number of vectors $x \in (\F_2)^3$ such that $\dim (\calV^T x)=1$ & $2$ & $2$ & $1$ & $1$ \\
\hline
\end{tabular}
\end{center}
\end{proof}

The proofs of Propositions \ref{J3dim2} and \ref{J3dim4} are similar and we shall leave them to the reader.
Let us only explain why the four generic matrices given in Proposition \ref{J3dim4} yield pairwise inequivalent
matrix spaces. We simply look at the structure of the sets of their rank $1$ matrices.
\begin{itemize}
\item If $\calV$ is equivalent to the space associated with $\mathbf{N}_1$, then it contains exactly two rank $1$ matrices.
\item If $\calV$ is equivalent to the space associated with $\mathbf{N}_2$, then it contains exactly three rank $1$ matrices,
and they have the same range.
\item If $\calV$ is equivalent to the space associated with $\mathbf{N}_3$, then it contains exactly three rank $1$ matrices,
and they do not have the same range.
\item Otherwise, $\calV$ contains a sole rank $1$ matrix.
\end{itemize}

\section{Applications}

\subsection{Triples of locally linearly dependent operators over $\F_2$}\label{LLDsection}

In \cite[Section 3]{dSPLLD2}, we have shown how minimal reduced locally linearly dependent operator spaces are connected to semi-primitive operator spaces.
Let us recall the basics: Let $U$ and $V$ be finite-dimensional vector spaces, and $\calS$ be a reduced linear subspace of $\calL(U,V)$.
We define the \textbf{dual operator space} $\widehat{\calS}$ of $\calS$ as the space of all operators from $\calS$ to $V$ of the form
$$\widehat{x} : f \in \calS \mapsto f(x), \quad \text{with $x \in U$.}$$
We say that $\calS$ is locally linearly dependent (in abbreviated form: LLD) when, for every vector $x \in U$, there is a non-zero operator $s \in \calS$ such that $s(x)=0$.
Then, $\calS$ is a minimal LLD space if and only if $\widehat{\calS}$ is semi-primitive.
Moreover, two reduced operator spaces $\calS$ and $\calT$ are equivalent if and only if their dual operator spaces are equivalent.
Noting that $\calS$ is always equivalent to $\widehat{\widehat{\calS}}$, this yields a one-to-one correspondence between
the equivalence classes of semi-primitive operator spaces and the ones of minimal reduced LLD spaces.

We have seen that the semi-primitive subspaces of $\calL(U,V)$ with upper rank $2$
are the primitive ones for which $\dim V>2$. Thus, we deduce the following result from Theorem \ref{classrank2F2}
and from Propositions \ref{J3dim2}, \ref{J3dim3} and \ref{J3dim4}.

\begin{theo}[Classification of $3$-dimensional minimal LLD spaces over $\F_2$]
Let $\calS \subset \calL(U,V)$ be a $3$-dimensional minimal reduced LLD space over $\F_2$.
Then, one and only one of the following situations holds:
\begin{enumerate}[(a)]
\item $\dim V=2$;
\item $\dim U=2$, $\dim V=3$ and $\calS$ is represented by the matrix space associated with
$$\begin{bmatrix}
\mathbf{x} & 0 \\
\mathbf{y} & \mathbf{y} \\
0 & \mathbf{z}
\end{bmatrix}.$$
\item $\dim U=3$, $\dim V=3$, and one and only one of the following generic matrices is associated with a matrix space that represents $\calS$:
$$\begin{bmatrix}
0 & -\mathbf{x} & -\mathbf{y} \\
\mathbf{x} & 0 & -\mathbf{z} \\
\mathbf{y} & \mathbf{z} & 0
\end{bmatrix}, \quad \begin{bmatrix}
0 & \mathbf{x} & \mathbf{x}+\mathbf{z} \\
\mathbf{x} & 0 & \mathbf{y} \\
\mathbf{x}+\mathbf{y} &  \mathbf{z}  & 0
\end{bmatrix},$$
$$\begin{bmatrix}
\mathbf{x} & 0 & \mathbf{z} \\
\mathbf{y} & \mathbf{y} & 0 \\
0 & \mathbf{z} & 0
\end{bmatrix},
\quad \begin{bmatrix}
\mathbf{x} & 0 & \mathbf{y} \\
\mathbf{y}+\mathbf{z} & \mathbf{y} & 0 \\
0 & \mathbf{z} & 0
\end{bmatrix}, \quad \begin{bmatrix}
\mathbf{x} & \mathbf{y} & 0 \\
\mathbf{y} & \mathbf{y} & \mathbf{z} \\
0 & \mathbf{z} & 0
\end{bmatrix}, \quad
\begin{bmatrix}
\mathbf{x} & 0 & \mathbf{y} \\
\mathbf{y} & \mathbf{y} & \mathbf{z} \\
0 & \mathbf{z} & 0
\end{bmatrix}.$$

\item $\dim U=4$, $\dim V=3$, and one and only one of the following generic matrices is associated with a matrix space that represents $\calS$:
$$\begin{bmatrix}
\mathbf{x} & 0 & \mathbf{y} &  0 \\
\mathbf{y} & \mathbf{y} & 0  & \mathbf{z} \\
0 & \mathbf{z} & 0  & 0
\end{bmatrix}, \quad
\begin{bmatrix}
\mathbf{x} & 0 & \mathbf{y} & \mathbf{z} \\
\mathbf{y} & \mathbf{y} & 0  & 0 \\
0 & \mathbf{z} & 0  & 0
\end{bmatrix}, \quad
\begin{bmatrix}
\mathbf{x} & 0 & \mathbf{z} & 0 \\
\mathbf{y} & \mathbf{y} & 0  & \mathbf{z} \\
0 & \mathbf{z} & 0  & 0
\end{bmatrix}, \quad
\begin{bmatrix}
\mathbf{x} & 0 & \mathbf{y} & \mathbf{z} \\
\mathbf{y} & \mathbf{y} & \mathbf{z}  & 0 \\
0 & \mathbf{z} & 0  & 0
\end{bmatrix}$$
and
$$\begin{bmatrix}
\mathbf{y} & 0 & \mathbf{z} & \mathbf{z} \\
0 & \mathbf{z} & \mathbf{x} & 0 \\
\mathbf{x} & \mathbf{x} & 0 & \mathbf{y}
\end{bmatrix}.$$

\item $\dim U=5$, $\dim V=3$, and $\calS$ is represented by the matrix space associated with the generic matrix
$$\begin{bmatrix}
\mathbf{x} & 0 & \mathbf{y} & \mathbf{z} & 0 \\
\mathbf{y} & \mathbf{y} & 0 & 0 & \mathbf{z} \\
0 & \mathbf{z} & 0 & 0 & 0
\end{bmatrix}.$$

\end{enumerate}
Conversely, all the above cited matrix spaces represent $3$-dimensional minimal LLD operator spaces.
\end{theo}

In (c), the given matrix spaces represent the dual operator spaces of the matrix spaces
$\Mata_3(\F_2)$, $\calU_3(\F_2)$, and the four matrix spaces cited in Proposition \ref{J3dim3}.
In (d), the given matrix spaces represent the dual operator spaces of the four matrix spaces cited in Proposition \ref{J3dim4}
and of $\calV_3(\F_2)$. In (e), the matrix space represents the dual operator space of $\calJ_3(\F_2)$.

The computation of the dual operator spaces is performed in the same way as in the proof of Lemma \ref{U3trans}.

\subsection{Subspaces of $\Mat_3(\F_2)$ with trivial spectrum}\label{trivialspectrumsection}

\begin{Def}
Given a field $\K$, a linear subspace $\calV$ of $\Mat_n(\K)$ is said to have a \textbf{trivial spectrum} when
no matrix of $\calV$ has an eigenvalue in $\K \setminus \{0\}$.
\end{Def}

In \cite{Quinlan} and \cite{dSPlargerank}, it was proved that a trivial spectrum subspace $\calV$ of $\Mat_n(\K)$ has dimension at most $\dbinom{n}{2}$.
In \cite{dSPlargeaffine}, the classification of trivial spectrum subspaces with the maximal dimension
was achieved for all fields with more than $3$ elements, and it was shown that the classification theorem failed for $\F_2$.
Our aim here is to use the classification of semi-primitive subspaces of $\Mat_3(\F_2)$ to obtain the full classification
of $3$-dimensional trivial spectrum subspaces of $\Mat_3(\F_2)$.

First of all, we introduce some notation:
\begin{Not}
Let $\calA$ and $\calB$ be linear subspaces, respectively, of $\Mat_n(\K)$ and $\Mat_p(\K)$. One denotes by $\calA \vee \calB$
the space of all matrices of the form
$$\begin{bmatrix}
A & C \\
[0]_{p \times n} & B
\end{bmatrix} \quad \text{with $A \in \calA$, $B \in \calB$ and $C \in \Mat_{n,p}(\K)$.}$$
\end{Not}

A trivial spectrum subspace $\calV$ of $\Mat_n(\K)$ is called \textbf{irreducible} when there is no proper and non-zero linear subspace
$F$ of $\K^n$ such that $\calV X \subset F$ for all $X \in F$. If the contrary holds we say that $\calV$ is reducible.
We have shown in \cite{dSPlargeaffine} that if a trivial spectrum subspace $\calV$ of $\Mat_n(\K)$ has dimension $\dbinom{n}{2}$,
then there is a list $(n_1,\dots,n_p)$ of positive integers such that $\sum_{k=1}^p n_k=n$, together with
irreducible trivial spectrum subspaces $\calV_1 \subset \Mat_{n_1}(\K), \; \calV_2 \subset \Mat_{n_2}(\K), \dots,\calV_p \subset \Mat_{n_p}(\K)$,
such that
$$\calV \simeq \calV_1 \vee \calV_2 \vee \cdots \vee \calV_p.$$
In order to obtain the structure of trivial spectrum spaces with the maximal dimension, it is therefore essential to classify
the irreducible ones up to similarity. The following result was obtained in \cite{dSPlargeaffine}:

\begin{theo}
Assume that $\# \K>2$.
The irreducible subspaces of $\Mat_n(\K)$ with trivial spectrum and dimension $\dbinom{n}{2}$
are the spaces of the form $P\Mata_n(\K)$, where $P \in \GL_n(\K)$ is a non-isotropic matrix, i.e.\ the quadratic form $X \mapsto X^T PX$ is non-isotropic.
Two such spaces $P\Mata_n(\K)$ and $Q\Mata_n(\K)$ are similar if and only if there is a non-zero scalar $\lambda$ such that
$Q$ is congruent to $\lambda P$, that is $Q=\lambda RPR^T$ for some $R \in \GL_n(\K)$.
\end{theo}

For $\F_2$, this result holds for $n=2$ as well (see the proof in Section 4.1 of \cite{dSPlargeaffine}).
In that case, the result is simple: An irreducible subspace of $\Mat_2(\F_2)$ with dimension $1$ is spanned by a matrix $M \in \Mat_2(\F_2)$
with no eigenvalue in $\F_2$. As $X^2+X+1$ is the only irreducible polynomial of degree $2$ over $\F_2$,
there is only one such matrix $M$ up to similarity: The companion matrix $M=\begin{bmatrix}
0 & 1 \\
1 & 1
\end{bmatrix}$.

Using this together with Proposition 16 of \cite{dSPlargeaffine} - which holds for all fields -
we deduce the $3$-dimensional reducible trivial spectrum subspaces of $\Mat_3(\F_2)$:

\begin{prop}\label{reducible}
Set $C:=\begin{bmatrix}
0 & 1 \\
1 & 1
\end{bmatrix}$.
Up to similarity, there are three reducible $3$-dimensional trivial spectrum subspaces of $\Mat_3(\F_2)$:
$$\F_2 C \vee \{0\}, \quad \{0\} \vee \F_2 C \quad \text{and} \quad \NT_3(\F_2).$$
\end{prop}

Now, we turn to the $3$-dimensional irreducible trivial spectrum subspaces of $\Mat_3(\F_2)$.
We shall prove the following result:

\begin{theo}[Classification of $3$-dimensional irreducible subspaces of $\Mat_3(\F_2)$ with trivial spectrum]\label{classtrivialspectrum}
Up to similarity, there are exactly three irreducible $3$-dimensional subspaces of $\Mat_3(\F_2)$ with trivial spectrum:
$$\calT_1:=\Biggl\{\begin{bmatrix}
a & b & a \\
a & a & c \\
b+c & a & a
\end{bmatrix} \mid (a,b,c)\in (\F_2)^3\Biggr\},$$
$$\calT_2:=\Biggl\{\begin{bmatrix}
a & b & a \\
a & a & c \\
0 & a & a
\end{bmatrix} \mid (a,b,c)\in (\F_2)^3\Biggr\} \quad \text{and} \quad
\calT_3:=\Biggl\{\begin{bmatrix}
a & b & a \\
a & 0 & c \\
0 & a & a
\end{bmatrix} \mid (a,b,c)\in (\F_2)^3\Biggr\}.$$
\end{theo}

Let us start by proving that $\calT_1$, $\calT_2$ and $\calT_3$ all satisfy the claimed properties and that they are pairwise unsimilar.
Remember that the identities $a(a+1)=0$ and $ab(a+b)=0$ hold for all $(a,b)\in (\F_2)^2$.
For all $(a,b,c)\in (\F_2)^3$, we compute
$$\begin{vmatrix}
a+1 & b & a \\
a & a+1 & c \\
b+c & a & a+1
\end{vmatrix}=(a+1)^3+bc(b+c)+a^3+ba(a+1)+ca(a+1)+(b+c)a(a+1)=a+1+a=1,$$
which shows that $\calT_1$ has a trivial spectrum. Similarly, for all $\varepsilon \in \F_2$ and all $(a,b,c)\in (\F_2)^3$, we have
$$\begin{vmatrix}
a+1 & b & a \\
a & \varepsilon a+1 & c \\
0 & a & a+1
\end{vmatrix}=(a+1)^2(\varepsilon a+1)+a^3+ca(a+1)+ba(a+1)=\varepsilon (a+1)a+(a+1)+a=1,$$
whence $\calT_2$ and $\calT_3$ have trivial spectra.

Moreover, we compute that, for all $(a,b,c)\in (\F_2)^3$,
$$\begin{vmatrix}
a & b & a \\
a & a & c \\
b+c & a & a
\end{vmatrix}=a^3+bc(b+c)+a^3+ba^2+ca^2+(b+c)a^2=0.$$
Therefore, every matrix in $\calT_1$ is singular. However, the matrix
$\begin{bmatrix}
1 & 0 & 1 \\
1 & 1 & 1 \\
0 & 1 & 1
\end{bmatrix}$ belongs to $\calT_2$ and has determinant $1$, whereas
$\begin{bmatrix}
1 & 1 & 1 \\
1 & 0 & 1 \\
0 & 1 & 1
\end{bmatrix}$ belongs to $\calT_3$ and has determinant $1$. Therefore, $\calT_1$ is unsimilar to both $\calT_2$ and $\calT_3$.
To see that $\calT_2$ and $\calT_3$ are unsimilar, we simply note that $\calT_3$ contains only trace zero matrices, whereas $\calT_2$
does not.

Finally, let us prove that $\calT_1$, $\calT_2$ and $\calT_3$ are all irreducible.
We have seen that if a $3$-dimensional trivial spectrum subspace of $\Mat_3(\F_2)$ is reducible, then it contains only singular matrices,
and it contains at least one rank $1$ matrix.
The spaces $\calT_2$ and $\calT_3$ are both irreducible because each one of them contains a non-singular matrix.
On the other hand, $\calT_1$ is irreducible because it contains no rank $1$ matrix: Indeed,
let $(a,b,c)\in (\F_2)^3 \setminus \{0\}$, and consider $M:=\begin{bmatrix}
a & b & a \\
a & a & c \\
b+c & a & a
\end{bmatrix}$. If $M$ has rank $1$, then it has trace $0$ since it cannot have a non-zero eigenvalue,
and this leads to $a=0$; Thus, $(b,c)\neq (0,0)$, and one sees that $M=\begin{bmatrix}
0 & b & 0 \\
0 & 0 & c \\
b+c & 0 & 0
\end{bmatrix}$ has rank $2$ because exactly two scalars among $b,c,b+c$ equal $1$.

\vskip 3mm
Next, we prove that every irreducible $3$-dimensional trivial spectrum subspace of $\Mat_3(\F_2)$
is similar to one of the $\calT_i$'s.
To achieve this, we shall use a new technique, featured in \cite{dSPAtkinsontoGerstenhaber},
that relates such subspaces to semi-primitive matrix spaces. We recall the basics now:
Let $\K$ be an arbitrary field and
$\calV$ be an irreducible trivial spectrum subspace of $\Mat_n(\K)$.
For each vector $X \in \K^n$, we obtain a bilinear form
$$(N,Y)\in \calV \times \K^n \longmapsto Y^T NX.$$
Choosing respective bases of $\calV$ and $\K^n$, we denote by $\calM$ the space of all matrices representing
the above bilinear forms in those bases.
Using the fact that $\calV$ is an irreducible trivial spectrum space, one obtains that
$\calM$ is reduced with upper rank less than $n$ and with dimension $n$.
If $\calM$ is not semi-primitive, if $\dim \calV=\dbinom{n}{2}$
and if, for all integers $p \in  \lcro 1,n-1\rcro$ and $m>1$,
the existence of a semi-primitive subspace of $\Mat_{m,p}(\K)$ implies $m\leq \dbinom{p}{2}$,
then the chain of arguments from Section 5 of \cite{dSPAtkinsontoGerstenhaber} yields that
$\calV$ is reducible, contradicting our assumptions.
In particular, we know from the classification of spaces of matrices with rank at most $1$ that a semi-primitive
subspace of $\Mat_{m,2}(\K)$ exists only if $m=1$, and that there is no semi-primitive subspace of $\Mat_{m,1}(\K)$.
Moreover, we have seen in the present article that the existence of a semi-primitive subspace of $\Mat_{m,3}(\K)$ implies that
$m\leq 3$. It follows that $\calM$ is semi-primitive whenever $n\leq 4$ and $\dim \calV=\dbinom{n}{2}$.

Now, we assume that $n=3$, $\dim \calV=3$ and $\K=\F_2$. Then, $\calM$ is a semi-primitive subspace of $\Mat_3(\F_2)$ with dimension $3$ and upper rank $2$,
and hence we deduce from Proposition \ref{basicsemiprimitiveprop}
that it is primitive. Thus, $\calM^T$ is also primitive with upper rank $2$.
One sees that $\calM^T$ represents the dual operator space $\widehat{\calV}$, whence
$\calV$ is equivalent to $\widehat{\calM^T}$.

From there, Theorem \ref{classrank2F2} yields key information on the structure of $\calM^T$, and hence on that of $\widehat{\calM^T}$.
Using that information will be of great help to understand the structure of $\calV$.
We distinguish between several cases.

\subsubsection{Case 1. $\calM^T$ is equivalent to $\Mata_3(\F_2)$.}

Then, $\widehat{\calM^T}$ is also equivalent to $\Mata_3(\F_2)$,
yielding a non-singular matrix $P \in \GL_3(\F_2)$ such that $\calV=P\,\Mata_3(\F_2)$ (remember that $Q^T \Mata_3(\F_2) Q=\Mata_3(\F_2)$
for all $Q \in \GL_3(\F_2)$). As every $3$-dimensional quadratic form over a finite field is isotropic, Proposition
10 of \cite{dSPlargeaffine} yields that $\calV$ cannot have a trivial spectrum, contradicting our assumptions.

\subsubsection{Case 2. $\calM^T$ is equivalent to $\calU_3(\F_2)$.}

As we have seen in the course of the proof of Lemma \ref{U3trans},
the dual operator space $\widehat{\calU_3(\F_2)}$ is equivalent to $\calU_3(\F_2)$, whence $\calV$ is equivalent to $\calU_3(\F_2)$.
In that case, we note that every matrix in $\calV$ has rank $2$ and that the matrices of $\calV$ have pairwise distinct ranges
and pairwise distinct kernels.

For $M \in \calV$, we write the characteristic polynomial of $M$ as $\chi_M(t)=t^3-\tr(M) t^2+q(M) t$,
and the condition that $1$ does not belong to the spectrum of $M$ reads $1-\tr(M)+q(M)=1$, whence
$$\forall M \in \calV, \; q(M)=\tr(M).$$
However, $q$ is a quadratic form with polar form $(A,B) \mapsto \tr(A)\tr(B)+\tr(AB)$, and hence
\begin{equation}\label{finalidentity}
\forall (A,B)\in \calV^2, \; \tr(AB)=\tr(A)\tr(B).
\end{equation}
The linear subspace $\calH:=\{M \in \calV : \; \tr(M)=0\}$ has codimension at most $1$ in $\calV$ and consists only of nilpotent matrices.
If $\dim \calH=3$, then $\calV=\calH$ and Gerstenhaber's theorem (see \cite{dSPGerstenhaberskew,Serezhkin})
would yield that $\calV$ is reducible. Thus, $\dim \calH=2$.
Using identity \eqref{finalidentity}, we see that $\calH$ is included in the radical of
the symmetric bilinear form $(A,B)\mapsto \tr(AB)$ on $\calV^2$.

\begin{claim}\label{claimH1ouH2}
We define
$$\calH_1:=\Biggl\{\begin{bmatrix}
0 & a & 0 \\
0 & 0 & b \\
a+b & 0 & 0
\end{bmatrix} \mid (a,b)\in (\F_2)^2\Biggr\} \quad \text{and} \quad
\calH_2:=\Biggl\{\begin{bmatrix}
0 & 0 & a \\
0 & 0 & b \\
b & a & 0
\end{bmatrix}\mid (a,b)\in (\F_2)^2\Biggr\}.$$
Then, $\calH$ is equivalent to $\calH_1$ or to $\calH_2$.
\end{claim}

More generally, it can be shown that an irreducible subspace of nilpotent matrices of $\Mat_3(\F_2)$
is always equivalent to $\calH_1$ or to $\calH_2$.

\begin{proof}
We have just seen that $\forall (A,B)\in \calH^2, \; \tr(AB)=0$.
Take linearly independent matrices $A_1$ and $A_2$ in $\calH$.
We know that $A_1$ and $A_2$ are both rank $2$ nilpotent matrices with different kernels and different ranges.
We distinguish between two main cases, whether $\Ker A_2 \subset \Ker A_1^2$ holds or not. \\
\textbf{Case a.} $\Ker A_2 \subset \Ker A_1^2$. Then,
we see that we can conjugate $\calH$ with an invertible matrix so as to reduce the situation to the one where
$$A_1=\begin{bmatrix}
0 & 1 & 0 \\
0 & 0 & 1 \\
0 & 0 & 0
\end{bmatrix}
\quad \text{and} \quad
A_2=\begin{bmatrix}
? & 0 & ? \\
? & 0 & ? \\
? & 0 & ?
\end{bmatrix}.$$
Using $\tr(A_1A_2)=0$, we deduce that the entry of $A_2$ at the $(2,1)$-spot is zero. If the one at the $(3,1)$-spot were
zero, then the whole first column of $A_2$ would be zero since $A_2$ is nilpotent, contradicting the fact that $A_2$ has rank $2$.
Thus, if we denote by $(e_1,e_2,e_3)$ the standard basis of $(\F_2)^3$,
we see that $A_2e_1 \not\in \Ker A_1^2$ and $A_1A_2e_1=e_2$. Thus, we may now use $(e_1, e_2,A_2 e_1)$ as our new basis, thereby reducing
the situation to the one where
$$A_1=\begin{bmatrix}
0 & 1 & 0 \\
0 & 0 & 1 \\
0 & 0 & 0
\end{bmatrix}
\quad \text{and} \quad
A_2=\begin{bmatrix}
0 & 0 & a \\
0 & 0 & b \\
1 & 0 & c
\end{bmatrix} \quad \text{for some $(a,b,c)\in (\F_2)^3$.}$$
As $\tr A_2=0$, we have $c=0$.
The characteristic polynomial of $A_1+A_2$ then equals $t^3+at+(b+1)$, whence $a=0$ and $b=1$.
We conclude that $\calH$ is the space of all matrices of the form $\begin{bmatrix}
0 & x & 0 \\
0 & 0 & y \\
x+y & 0 & 0
\end{bmatrix}$ with $(x,y)\in (\F_2)^2$.

\textbf{Case b.} $\Ker A_2 \not\subset \Ker A_1^2$.
Then, we take $x \in \Ker A_2 \setminus \{0\}$ and we work with the basis $(A_1^2 x, A_1 x,x)$.
Thus, using the relations $\tr(A_2)=0$ and $\tr(A_1A_2)=0$, the situation is reduced to the one where
$$A_1=\begin{bmatrix}
0 & 1 & 0 \\
0 & 0 & 1 \\
0 & 0 & 0
\end{bmatrix}
\quad \text{and} \quad
A_2=\begin{bmatrix}
a & c & 0 \\
b & a & 0 \\
d & b & 0
\end{bmatrix} \quad \text{for some $(a,b,c,d)\in (\F_2)^4$.}$$
The characteristic polynomial of $A_2+A_1$ is
$t^3+t(a+bc)+(ab+d(c+1))$, and hence $a=bc$ and $ab=d(c+1)$.
Note that $(d,b)\neq (0,0)$ since $\im A_2 \neq \im A_1$ and $\rk A_2=\rk A_1$.
If $b=0$, then we deduce that $d=1$, and hence $a=0$ and $c=1$. In that case, we see once more that
$\calH$ is similar to the same space as in Case a.
Assume now that $b=1$. Then, $c=a=d(c+1)$, and hence $c=c^2=d(c+1)c=0$.
It follows that $a=0$ and $d=0$. Using the basis $(e_1,e_3,e_2)$, we conclude that $\calH$ is similar
to $\calH_2$.
\end{proof}

Now, we aim at discarding the second case in Claim \ref{claimH1ouH2}. Assume that $\calH$ is similar to $\calH_2$.
Then, no generality is lost in assuming that $\calH=\calH_2$.
As no matrix $A \in \calV$ satisfies $Ae_3=e_3$, we have $\dim \calV e_3\leq 2$, yielding a non-zero matrix $M \in \calV$ such that $M e_3=0$.
Then, $M \not\in \calH$. Let us write
$$M=\begin{bmatrix}
N & [0]_{2 \times 1} \\
L & 0
\end{bmatrix} \quad \text{with $N \in \Mat_2(\F_2)$ and $L \in \Mat_{1,2}(\F_2)$.}$$
As $\tr(MA)=0$ for all $A \in \calH$, we find that $L=0$.
As $M+A$ is singular for all $A \in \calH$, computing the determinant shows that, for $K:=\begin{bmatrix}
0 & 1 \\
1 & 0
\end{bmatrix}$, one has $X^T NKX=0$ for all $X \in (\F_2)^2$. One deduces that $NK$ is alternating, and hence $N=I_2$ as $N$ is non-zero.
We obtain that $\tr(M)=0$, contradicting the assumption that $M \not\in \calH$.

Thus, $\calH$ is similar to $\calH_1$. Without loss of generality, we may assume that $\calH=\calH_1$.
Now, let us choose a matrix $M$ in $\calV \setminus \calH$. Adding an appropriate matrix of $\calH$, we may assume that
$$M=\begin{bmatrix}
? & 0 & ? \\
? & ? & 0 \\
? & ? & ?
\end{bmatrix}.$$
As $\tr(MA)=0$ for all $A \in \calH$, while $\tr(M)=1$, we obtain $(a,b,c,d)\in (\F_2)^4$ such that
$$M=\begin{bmatrix}
a & 0 & d \\
d & a+b+1 & 0 \\
c & d & b
\end{bmatrix}.$$
If $d=0$, then we see that $a$, $b$ and $a+b+1$ are all eigenvalues of $M$, which is impossible since not all of them are zero.
It follows that $d=1$. Then, for all $(x,y)\in (\F_2)^2$, we deduce that
$$0=\begin{vmatrix}
a & x & 1 \\
1 & a+b+1 & y \\
c+x+y & 1 & b
\end{vmatrix}=cxy+(a+1)x+(b+1)y+\bigl((a+b+1)c+ab+1\bigr).$$
It follows that $c=0$, $a=1$ and $b=1$. Thus, $\calV=\calT_1$.

\subsubsection{Case 3. $\calM^T$ is equivalent to a subspace of $\calJ_3(\F_2)$.}

As $\calM^T$ has dimension $3$, we lose no generality in assuming that it equals one of the spaces
listed in Proposition \ref{J3dim3}.
To the space of all operators $M \in \calM^T \mapsto Mx$, with $x \in (\F_2)^2 \times \{0\}$,
then corresponds a linear subspace $\calH$ of $\calV$ with one of the following properties,
whether $\calM^T$ is represented by one of the generic matrices $\mathbf{M}_2,\mathbf{M}_4$ or by one of the generic matrices
$\mathbf{M}_1,\mathbf{M}_3$:
\begin{itemize}
\item Subcase 3.1. There is a basis $(A_1,A_2)$ of $\calH$ in which $\rk A_1=1$, $\rk A_2=2$, $\im A_1 \subset \im A_2$ and $\Ker A_1\oplus \Ker A_2=(\F_2)^3$.
\item Subcase 3.2. $\calH$ contains two rank $1$ matrices $A_1$ and $A_2$ such that $\im A_1 \neq \im A_2$ and $\Ker A_1 \neq \Ker A_2$.
\end{itemize}

Before we can tackle each case separately, we need a simple lemma:

\begin{lemma}\label{imagelemma}
Let $\calT$ be a trivial spectrum linear subspace of $\Mat_3(\F_2)$ in which all the elements have their range included in $(\F_2)^2 \times \{0\}$.
Then, there is a matrix $N \in \Mat_2(\K)$ whose spectrum does not contain $1$ and such that every matrix of $\calT$ splits up as
$$M=\begin{bmatrix}
\lambda N & [?]_{2 \times 1} \\
[0]_{1 \times 2} & 0
\end{bmatrix}\quad \text{for some $\lambda \in \F_2$.}$$
\end{lemma}

\begin{proof}
We can write every matrix of $\calT$ as
$$M=\begin{bmatrix}
K(M) & [?]_{2 \times 1} \\
[0]_{1 \times 2} & 0
\end{bmatrix}\quad \text{with $K(M) \in \Mat_2(\F_2)$.}$$
Then, $K(\calH)$ is a trivial spectrum subspace of $\Mat_2(\F_2)$.
By Theorem 9 of \cite{dSPlargerank}, we have $\dim K(\calH) \leq 1$. The result follows by taking $N$ as the sole non-zero vector of $K(\calH)$
if $\dim K(\calH)=1$, and $N=0$ otherwise.
\end{proof}

We seek to discard Subcase 3.1. Assume that it holds and note that the $2$-dimensional space $\im A_2$ contains the range of every matrix in $\calH$.
Conjugating $\calV$ with a well-chosen invertible matrix, we lose no generality in assuming that $\im A_2=\K^2 \times \{0\}$.
The above lemma yields some $N \in \Mat_2(\F_2)$ for which $1$ is not an eigenvalue and such that every matrix $M$ of $\calH$ splits up as
$$M=\begin{bmatrix}
\lambda N & [?]_{2 \times 1} \\
[0]_{1 \times 2} & 0
\end{bmatrix}\quad \text{for some $\lambda \in \F_2$.}$$
If $N$ were singular, we would find a non-zero vector that belongs to the kernel of all the matrices in $\calH$, contradicting
the fact that $\Ker A_1 \oplus \Ker A_2=(\F_2)^3$. Therefore, $N \in \GL_2(\F_2)$, and hence the
characteristic polynomial of $N$ must be $t^2+t+1$.
As $A_2$ has rank $2$, we must have
$$A_2=\begin{bmatrix}
N & [?]_{2 \times 1} \\
[0]_{1 \times 2} & 0
\end{bmatrix}.$$
In turn, this shows that $\Ker A_2 \oplus \im A_2=(\F_2)^3$, whence
an additional conjugation allows one to assume that $\Ker A_2$ is the span of $\begin{bmatrix}
0 & 0 & 1
\end{bmatrix}^T$. On the other hand, as $A_1$ has rank $1$, we must have
$$A_1=\begin{bmatrix}
[0]_{2 \times 2} & C \\
[0]_{1 \times 2} & 0
\end{bmatrix} \quad \text{for some $C \in (\F_2)^2 \setminus \{0\}$.}$$
Since $N$ has no eigenvalue in $\F_2$, we see that $C$ and $NC$ are linearly independent, and we note that $N(NC)=C+NC$.
Conjugating by the change of bases matrix
$$Q:=\begin{bmatrix}
C & NC & [0]_{2 \times 1} \\
0 & 0 & 1
\end{bmatrix} \in \GL_3(\F_2),$$
we reduce the situation further to the point where
$$A_1=\begin{bmatrix}
0 & 0 & 1 \\
0 & 0 & 0 \\
0 & 0 & 0
\end{bmatrix} \quad \text{and} \quad A_2=\begin{bmatrix}
0 & 1 & 0 \\
1 & 1 & 0 \\
0 & 0 & 0
\end{bmatrix}.$$
Then, we extend $(A_1,A_2)$ into a basis $(A_1,A_2,A_3)$ of $\calV$.
Choosing $A_3$ well, we may assume that
$$A_3=\begin{bmatrix}
a & 0 & 0 \\
b & c & d \\
e & f & g
\end{bmatrix} \quad \text{for some $(a,b,c,d,e,f,g)\in (\F_2)^7$.}$$
Note that $a=0$ since $A_3$ has no non-zero eigenvalue.
Denote by $(e_1,e_2,e_3)$ the standard basis of $(\F_2)^3$.
Recall that, for all non-zero vectors $x \in (\F_2)^3$, none of the spaces $\calV^T x$ and $\calV x$ contains
$x$, whence $\dim \calV^T x<3$ and $\dim \calV x<3$.
In particular, $\dim \calV(e_2+e_3)<3$ and $\dim \calV(e_1+e_3)<3$ yield $f=g$ and $e=g$, respectively.
As $\calV$ is irreducible, we deduce that $e=f=g=1$.
Then we obtain $\calV^T(e_1+e_3)=(\F_2)^3$, contradicting the above remarks.

\vskip 3mm
We have just shown that Subcase 3.1 cannot hold. Thus, we obtain two rank $1$ matrices $A_1$ and $A_2$ in $\calV$
with distinct kernels and ranges. Setting $P:=\im A_1+\im A_2$, we lose no generality in assuming that $P=(\F_2)^2 \times \{0\}$.
Let us consider a matrix $N \in \Mat_2(\F_2)$ obtained by applying Lemma \ref{imagelemma} to the trivial spectrum space
$\calH:=\Vect(A_1,A_2)$. As $\Ker A_1 \neq \Ker A_2$, we must have $N \neq 0$,
and, without loss of generality, we may assume that
$$A_1=\begin{bmatrix}
N & [?]_{2 \times 1} \\
[0]_{1 \times 2} & 0
\end{bmatrix}.$$
As $A_1$ has rank $1$ and has no non-zero eigenvalue in $\F_2$, it is nilpotent, whence $N$ is nilpotent.
Thus, as $A_1$ has rank $1$, no further generality is lost in assuming that
$$N=\begin{bmatrix}
0 & 1 \\
0 & 0
\end{bmatrix} \quad \text{and} \quad
A_1=\begin{bmatrix}
0 & 1 & \alpha \\
0 & 0 & 0 \\
0 & 0 & 0
\end{bmatrix} \quad \text{for some $\alpha \in \F_2$.}$$
As $A_2$ has rank $1$ and $\im A_2 \neq \im A_1$, the only option is that
$$A_2=\begin{bmatrix}
0 & 0 & \beta \\
0 & 0 & 1 \\
0 & 0 & 0
\end{bmatrix} \quad \text{for some $\beta \in \F_2$.}$$
From there, conjugating $\calV$ with $Q:=\begin{bmatrix}
1 & -\beta & 0 \\
0 & 1 & \alpha \\
0 & 0 & 1
\end{bmatrix}$ reduces the situation to the one where
$$\calH=\Biggl\{
\begin{bmatrix}
0 & x & 0 \\
0 & 0 & y \\
0 & 0 & 0
\end{bmatrix}\mid (x,y)\in (\F_2)^2\Biggr\}.$$
Now, we find a matrix $A_3$ of $\calV \setminus \calH$ of the form
$$A_3=\begin{bmatrix}
a & 0 & b \\
c & d & 0 \\
e & f & g
\end{bmatrix} \quad \text{with $(a,b,c,d,e,f,g)\in (\F_2)^7$.}$$
Noting that $\calV=\Vect(A_1,A_2,A_3)$, we find $1=\begin{vmatrix}
a+1 & x & b \\
c & d+1 & y \\
e & f & g+1
\end{vmatrix}$ for all $(x,y)\in (\F_2)^2$. Expanding the right-hand side of this equality, we find a polynomial with degree
at most $1$ in each variable $x$ and $y$, of the form
$$e(xy)+c(g+1)x+f(a+1)y+?.$$
It follows that $e=0$, $c(g+1)=0$ and $f(a+1)=0$.
If $c=0$, then $a$ is an eigenvalue of $A_3$, and hence the first column of $A_3$ is zero. It would follow that $\calV$ is reducible, contradicting
our assumptions. Thus, $c=1$, and one finds $f=1$ with the same line of reasoning.
One deduces from the above equalities that $g=1$ and $a=1$. As $1$ is not an eigenvalue of $A_3$, we must have $b=1$.
Finally, we conclude that $\calV=\calT_2$ or $\calV=\calT_3$, whether $d=1$ or $d=0$. This finishes the proof of Theorem \ref{classtrivialspectrum}.

\begin{Rem}
Using the elementary operation $L_2 \leftarrow L_2+L_3$, one sees that $\calT_2$ and $\calT_3$ are equivalent.
Moreover, one can check that the dual operator space of $\calT_3$ is equivalent to the matrix space associated with
$\mathbf{M}_3$.
\end{Rem}

\subsection{Affine subspaces of non-singular matrices of $\Mat_3(\F_2)$}\label{affinenonsingularsection}

Using the classification of $3$-dimensional trivial spectrum subspaces of $\Mat_3(\F_2)$, we are now able
to classify, up to equivalence, the $3$-dimensional affine subspaces of $\Mat_3(\F_2)$ that are included in $\GL_3(\F_2)$.

We need only classify those affine subspaces that contain $I_3$. Given a linear subspace $H$ of $\Mat_3(\F_2)$,
the affine subspace $I_3+H$ is included in $\GL_3(\F_2)$ if and only if $H$ is a trivial spectrum space.
Note that if $H$ and $H'$ are similar linear subspaces, then $I_3+H$ and $I_3+H'$ are equivalent affine spaces
(the converse does not hold in general).
Therefore, using Proposition \ref{reducible} and Theorem \ref{classtrivialspectrum}, we obtain that each $3$-dimensional affine subspace that is included in
$\GL_3(\F_2)$ is equivalent to one of the following six affine subspaces, where we have set
$C:=\begin{bmatrix}
0 & 1 \\
1 & 1
\end{bmatrix}$:
$$I_3+\NT_3(\F_2), \quad I_3+(\F_2 C \vee \{0\}), \quad I_3+(\{0\} \vee \F_2 C), \quad
I_3+\calT_1, \quad  I_3+\calT_2, \quad  \text{and} \quad I_3+\calT_3.$$
It remains to investigate potential equivalences between those six spaces.
This involves the following lemma:

\begin{lemma}\label{irreduciblelemma}
Let $\K$ be an arbitrary field, and $H_1$ and $H_2$ be linear subspaces of $\Mat_n(\K)$ for which the
affine spaces $I_n+H_1$ and $I_n+H_2$ are equivalent. Then, $H_1$ is irreducible if and only if $H_2$ is irreducible.
\end{lemma}

\begin{proof}
Assume that $H_1$ is reducible, and consider a non-zero proper linear subspace $P$ of $\K^n$ such that
$Mx \in P$ for all $M \in H_1$ and $x \in P$. It follows that
$Mx \in P$ for all $M \in (I_n+H_1)$ and $x \in P$. Using the assumed equivalence between $I_n+H_1$ and $I_n+H_2$,
we recover non-zero proper linear subspaces $P_1$ and $P_2$ of $\K^n$ such that
$Mx \in P_2$ for all $M \in (I_n+H_2)$ and $x \in P_1$.
In particular, this holds for $M=I_n$, whence $P_2=P_1$. Thus, we conclude that $Mx=(I_n+M)x-I_n x \in P_1$ for all $x \in P_1$ and all $M \in H_2$,
whence $H_2$ is reducible. Symmetrically, one obtains that $H_1$ is reducible whenever $H_2$ is reducible.
\end{proof}

Using the above lemma, we deduce that any one of the spaces $I_3+\NT_3(\F_2)$, $I_3+(\F_2 C \vee \{0\})$, $I_3+(\{0\} \vee \F_2 C)$
is inequivalent to any one of the spaces $I_3+\calT_i$ for $i \in \{1,2,3\}$.
Moreover, by Proposition 17 of \cite{dSPlargeaffine}, the spaces $I_3+\NT_3(\F_2)$, $I_3+(\F_2 C \vee \{0\})$ and $I_3+(\{0\} \vee \F_2 C)$
are pairwise inequivalent.
It remains only to investigate possible equivalences between the spaces $I_3+\calT_1$, $I_3+\calT_2$ and $I_3+\calT_3$.

Note that equivalent affine spaces have equivalent translation vector spaces. However, $\calT_1$ is inequivalent to both $\calT_2$ and $\calT_3$
as $\calT_1$ contains only rank $2$ matrices, whereas $\calT_2$ and $\calT_3$ both contain rank $1$ matrices.

Finally, we show that $I_3+\calT_2$ is equivalent to $I_3+\calT_3$.
To see this, we choose an arbitrary matrix $A \in \calT_3 \cap \GL_3(\F_2)$
(an obvious choice is $A=\begin{bmatrix}
1 & 0 & 1 \\
1 & 0 & 0 \\
0 & 1 & 1
\end{bmatrix}$), so that the characteristic polynomial of $A$ is $t^3+t+1$.
Thus, the characteristic polynomial of $I_3+A$ is $t^3+t^2+1$, so that $\tr\bigl((I_3+A)^{-1}\bigr)=0$.
We note that
$$I_3+\calT_3=(I_3+A)+\calT_3=(I_3+A)\bigl(I_3+(I_3+A)^{-1}\calT_3\bigr) \sim I_3+(I_3+A)^{-1}\calT_3.$$
Using Lemma \ref{irreduciblelemma}, we see that $(I_3+A)^{-1}\calT_3$ is an irreducible trivial spectrum space with dimension $3$;
As it is equivalent to $\calT_3$, it cannot be equivalent to $\calT_1$, whence it is similar to $\calT_2$ or to $\calT_3$.
However, since $\tr((I_3+A)^{-1}A)=\tr(I_3)+\tr((I_3+A)^{-1})=1$, we see that $(I_3+A)^{-1}\calT_3$
contains a matrix with trace $1$, whence $(I_3+A)^{-1}\calT_3$ is unsimilar to $\calT_3$.
We conclude that $(I_3+A)^{-1}\calT_3$ is similar to $\calT_2$, whence $I_3+\calT_2 \sim I_3+\calT_3$.
Let us sum up our results:

\begin{theo}[Classification of $3$-dimensional affine subspaces of non-singular matrices of $\Mat_3(\F_2)$]
Set $C:=\begin{bmatrix}
0 & 1 \\
1 & 1
\end{bmatrix}$.
Up to equivalence, exactly five $3$-dimensional affine subspaces of $\Mat_3(\F_2)$ are included in
$\GL_3(\F_2)$: They are the ones which contain $I_3$ and with respective translation vector spaces
$\NT_3(\F_2)$, $\F_2 C \vee \{0\}$, $\{0\} \vee \F_2 C$,
$\calT_1$ and $\calT_2$.
\end{theo}


\begin{thebibliography}{1}
\bibitem{AtkinsonPrim}
M. D. Atkinson,
{Primitive spaces of matrices of bounded rank II.}
J. Austral. Math. Soc. (Ser. A)
{\bf 34} (1983), 306--315.


\bibitem{AtkLloydPrim}
M. D. Atkinson, S. Lloyd,
{Primitive spaces of matrices of bounded rank.}
J. Austr. Math. Soc. (Ser. A)
{\bf 30} (1980), 473--482.

\bibitem{Beasley}
L. B. Beasley,
{Spaces of rank $2$ matrices over $GF(2)$.}
Electron. J. Linear Algebra
{\bf 5} (1999), 11--18.


\bibitem{EisenbudHarris}
D. Eisenbud, J. Harris,
{Vector spaces of matrices of low rank.}
Adv. Math.
{\bf 37}(1988), 135--155.



\bibitem{FillmoreLaurieRadjavi}
P. Fillmore, C. Laurie, H. Radjavi,
{On matrix spaces with zero determinant.}
Lin. Multilin. Algebra
{\bf 18-3} (1985), 255--266.

\bibitem{Quinlan}
R. Quinlan,
{Spaces of matrices without non-zero eigenvalues in their field of definition, and a question of Szechtman.}
Linear Algebra Appl.
{\bf 434} (2011), 1580--1587.

\bibitem{dSPLLD2}
C. de Seguins Pazzis,
{Local linear dependence seen through duality II.}
Linear Algebra Appl.
{\bf 462}
(2014), 133--185.

\bibitem{dSPfeweigenvalues}
C. de Seguins Pazzis,
{Spaces of matrices with few eigenvalues.}
Linear Algebra Appl.
{\bf 449} (2014), 210--311.

\bibitem{dSPclass}
C. de Seguins Pazzis,
{The classification of large spaces of matrices with bounded rank.}
Israel J. Math.
in press (2015).

\bibitem{dSPAtkinsontoGerstenhaber}
C. de Seguins Pazzis,
{From primitive spaces of bounded rank matrices to a generalized Gerstenhaber theorem.}
Quart. J. Math.
{\bf 65-2} (2014), 319--325.

\bibitem{dSPlargeaffine}
C. de Seguins Pazzis,
{Large affine spaces of non-singular matrices.}
Trans. Amer. Math. Soc.
{\bf 365} (2013), 2569--2596.


\bibitem{invitquad}
C. de Seguins Pazzis,
{Invitation aux formes quadratiques.}
Calvage \& Mounet,
Paris, 2011.

\bibitem{dSPGerstenhaberskew}
C. de Seguins Pazzis,
{On Gerstenhaber's theorem for spaces of nilpotent matrices over a skew field.}
Linear Algebra Appl.
{\bf 438-11} (2013), 4426--4438.

\bibitem{dSPlargerank}
C. de Seguins Pazzis,
{On the matrices of given rank in a large subspace.}
Linear Algebra Appl.
{\bf 435-1} (2011), 147--151.

\bibitem{Serezhkin}
V. N. Serezhkin,
{Linear transformations preserving nilpotency (in Russian).}
Izv. Akad. Nauk BSSR, Ser. Fiz.-Mat. Nauk
{\bf 125} (1985), 46--50.

\end{thebibliography}
\end{document}